\documentclass[12pt,letterpaper,titlepage]{amsart}
\usepackage[in]{fullpage}
\usepackage[utf8]{inputenc}

\usepackage{amsmath, amssymb, amsthm, amsfonts,amscd,amsaddr,amsrefs}

\theoremstyle{plain}
\newtheorem{theorem}{Theorem}[section]
\newtheorem{proposition}[theorem]{Proposition}
\newtheorem{corollary}[theorem]{Corollary}
\newtheorem{lemma}[theorem]{Lemma}
\theoremstyle{definition}
\newtheorem{definition}[theorem]{Definition}
\newtheorem{remark}[theorem]{Remark}
\numberwithin{equation}{section}

\usepackage{enumitem}
\setlist[enumerate,2]{label=\textit{(\alph*)},ref=(\textit{\alph*})}
\setlist[enumerate,1]{label=\textit{(\roman*)},ref=(\textit{\roman*})}

\newenvironment{Cases}{\hspace{-8pt}\begin{cases}}{\end{cases}\hspace{-20pt}}

\newcounter{zeile}
\newcommand{\zz}{\stepcounter{zeile}(\arabic{zeile})}

\newcommand{\C}{\mathbb{C}}
\newcommand{\Nc}{\mathcal{N}}
\newcommand{\cM}{\mathcal{M}}
\newcommand{\Hb}{\mathbb{H}}
\newcommand{\Q}{\mathbb{Q}}
\newcommand{\R}{\mathbb{R}}
\newcommand{\N}{\mathbb{N}}
\newcommand{\Pb}{\mathbb{P}}
\newcommand{\SO}{\operatorname{SO}}
\newcommand{\Herm}{\operatorname{Herm}}
\newcommand{\GL}{\operatorname{GL}}
\newcommand{\SL}{\operatorname{SL}}
\newcommand{\SU}{\operatorname{SU}}
\newcommand{\tr}{\operatorname{tr}}
\newcommand{\im}{\operatorname{Im}}
\newcommand{\Sp}{\operatorname{Sp}}
\newcommand{\Ad}{\operatorname{Ad}}
\newcommand{\ad}{\operatorname{ad}}
\newcommand{\diag}{\operatorname{diag}}
\newcommand{\Res} {\operatorname{Res}} 
\newcommand{\Spin}{\operatorname{Spin}}
\newcommand{\Sym}{\operatorname{Sym}}
\newcommand{\Skew}{\operatorname{Skew}}
\newcommand{\rank}{\operatorname{rank}}
\newcommand{\re}{\operatorname{Re}}
\newcommand{\Mat}{\operatorname{Mat}}
\newcommand{\Gq}{{\oline G}}
\newcommand{\Hq}{{\oline H}}
\renewcommand\hat{\widehat}
\newcommand{\af}{\mathfrak{a}}
\newcommand{\e}{\epsilon}
\newcommand{\gf}{\mathfrak{g}}
\newcommand{\ff}{\mathfrak{f}}
\newcommand{\hf}{\mathfrak{h}}
\newcommand{\kf}{\mathfrak{k}}
\newcommand{\lf}{\mathfrak{l}}
\newcommand{\mf}{\mathfrak{m}}
\newcommand{\nf}{\mathfrak{n}}
\newcommand{\pf}{\mathfrak{p}}
\newcommand{\spin}{\mathfrak{spin}}
\newcommand{\qf}{\mathfrak{q}}
\renewcommand{\sf}{\mathfrak{s}}
\renewcommand{\sl}{\mathfrak{sl}}
\newcommand{\gl}{\mathfrak{gl}}
\newcommand{\symp}{\mathfrak{sp}}
\newcommand{\so}{\mathfrak{so}}
\renewcommand{\sp}{\mathfrak{sp}}
\newcommand{\su}{\mathfrak{su}}
\newcommand{\uf}{\mathfrak{u}}
\newcommand{\zf}{\mathfrak{z}}
\newcommand{\la}{\langle}
\newcommand{\ra}{\rangle}
\newcommand{\1}{{\bf1}}
\newcommand{\oline}{\overline}
\newcommand{\F}{\mathcal{F}}
\newcommand{\Unitary}{\operatorname{U}}
\newcommand{\fhq}{\overline{\mathfrak{h}}}
\newcommand{\HH}{\mathbb{H}} 
\newcommand{\sA}{\mathsf{A}}
\newcommand{\sB}{\mathsf{B}}
\newcommand{\sC}{\mathsf{C}}
\newcommand{\sE}{\mathsf{E}}
\newcommand{\sF}{\mathsf{F}}
\newcommand{\sG}{\mathsf{G}}
\newcommand{\aas}{\llap{$*$}}

\newcounter{class}
\newcommand{\yy}{\stepcounter{class}(\arabic{class})}
\newcounter{Tabelle}
\newcommand{\Tabelle}[1]{\refstepcounter{Tabelle}Table \arabic{Tabelle}\label{#1}}

\usepackage[usenames]{color}

\title[Classification of real spherical pairs]
{Classification of reductive real spherical pairs\\ I. The simple case}

\begin{document}
\date{October 24, 2017}

\begin{abstract} This paper gives a classification of all pairs $(\gf,\hf)$ 
with $\gf$ a simple real Lie algebra and $\hf\subset\gf$ a reductive subalgebra 
for which there exists a minimal parabolic subalgebra $\pf\subset\gf$ such that 
$\gf=\hf+\pf$ as vector sum. 
\end{abstract}

\author[Knop]{Friedrich Knop}
\email{friedrich.knop@fau.de}
\address{FAU Erlangen-Nürnberg, Department Mathematik\\
Cauerstr. 11, D-91058 Erlangen, Germany} 

\author[Krötz]{Bernhard Krötz}
\email{bkroetz@gmx.de}
\address{Universität Paderborn, Institut für Mathematik\\Warburger Straße 100, 
D-33098 Paderborn, Germany}

\author[Pecher]{Tobias Pecher} 
\email{tpecher@math.upb.de}
\address{Universität Paderborn, Institut für Mathematik\\Warburger Straße 100, 
D-33098 Paderborn, Germany}

\thanks{The second author was supported by ERC Advanced Investigators Grant HARG 268105}
\author[Schlichtkrull]{Henrik Schlichtkrull}
\email{schlicht@math.ku.dk}
\address{University of Copenhagen, Department of Mathematics\\Universitetsparken 5, 
DK-2100 Copenhagen \O, Denmark}

\maketitle

\section{Introduction}

\subsection{Spherical pairs}

We recall that a pair $(\gf_\C, \hf_\C)$ consisting of a complex
reductive Lie algebra $\gf_\C$ and a complex subalgebra $\hf_\C$
thereof is called {\it spherical} provided there exists a Borel
subalgebra ${\mathfrak b}_\C\subset\gf_\C$ such that
$\gf_\C=\hf_\C +{\mathfrak b}_\C$ as a sum of vector spaces (not
necessarily direct). In particular, this is the case for symmetric
pairs, that is, when $\hf_\C$ consists of the elements fixed by an
involution of $\gf_\C$.

Complex spherical pairs with $\hf_\C$ reductive were classified by
Krämer \cite{Kr} for $\gf_\C$ simple and for $\gf_\C$ semisimple by
Brion \cite{Brion} and Mikityuk \cite{Mik}.

The objective of this paper is to obtain the appropriate real version
of the classification of Krämer.  To be more precise, let $\gf$ be a
real reductive Lie algebra and $\hf\subset\gf$ a subalgebra. We call
$\hf$ {\it real spherical} provided there exists a minimal parabolic
subalgebra $\pf\subset\gf$ such that $\gf = \hf +\pf$.  Being in this
situation we call $(\gf,\hf)$ a {\it real spherical pair}.  The pair
is said to be trivial if $\hf=\gf$.

We say that $(\gf,\hf)$ is {\it absolutely spherical} if the
complexified pair $(\gf_\C,\hf_\C)$ is spherical.  It is easy to see
(cf.~Lemma \ref{real form}) that then $(\gf,\hf)$ is real spherical.
In particular, all real symmetric pairs $(\gf,\hf)$ are absolutely
spherical, since the involution of $\gf$ that defines $\hf$ extends to
an involution of $\gf_\C$.  The real symmetric pairs were classified
by Berger \cite{Berger}.  It is not difficult to classify also the
non-symmetric absolutely spherical pairs with $\hf$ reductive; this is
done in Table \ref{non-symmetric real forms} at the end of the paper.

\subsection{Main result}

Assume that $\gf$ is simple and non-compact.  The main result of this
paper is a classification of all reductive subalgebras of $\gf$ which
are real spherical. The following Table~\ref{table real spherical}
presents the most important outcome.  It contains all the real
spherical pairs which are not absolutely spherical, up to isomorphism
(and a few more, see Remark \ref{remark intro}).

Formally the classification is given in the following theorem, which
refers to a number of tables in addition to Table \ref{table real
  spherical}. These tables are collected at the end of the paper,
except for the above-mentioned list of Berger.

\begin{theorem}\label{theorem classification}  
  Let $(\gf,\hf)$ be a non-trivial real spherical pair for which $\gf$
  is simple and $\hf$ an algebraic and reductive subalgebra.  Then at
  least one of the following statements holds:
  \renewcommand{\theenumi}{\roman{enumi}}
  \begin{enumerate}
  \item\label{comp} $\gf$ is compact,
  \item\label{symm} $(\gf,\hf)$ is symmetric and listed by Berger (see
    \cite{Berger}*{Tableaux II}),
  \item\label{abs-sph} $(\gf,\hf)$ is absolutely spherical, but
    non-symmetric (see Tables \ref{non-symmetric spherical},
    \ref{complex forms}, and \ref{non-symmetric real forms}),
  \item\label{not abs-sph} $(\gf,\hf)$ is isomorphic to some pair in
    Table \ref{table real spherical}.
  \end{enumerate}
  \noindent Conversely, all pairs mentioned in
    \ref{comp}--\ref{not abs-sph} are real spherical.
\end{theorem}

\begin{table}[ht]
  %	\small
\[\begin{array}{rllll}
    &\gf&\hf\\
    \hline
    \yy&\su(p_1+p_2,q_1+q_2)&\aas\su(p_1,q_1)+\su(p_2,q_2)&&(p_1,q_1)\ne(q_2,p_2)\\
    \yy&\su(n,1)&\su(n-2q,1){+}\sp(q)+\ff&\ff\subset \uf(1)&1\le q\le\frac n2\\
    \yy&\sl(n,\HH)&\sl(n-1,\HH)+\ff&\ff\subset \C&n\ge3   \\
    \yy&\sl(n,\HH)&\aas\sl(n,\C)&&n\text{ odd}\\
    \hline
    \yy&\sp(p,q)&\aas\su(p,q)&&p\ne q\\
    \yy&\sp(p,q)&\aas\sp(p-1,q)&&p,q\ge1\\
    \hline
    \yy&\so(2p,2q)&\aas\su(p,q)&&p\ne q\\
    \yy&\so(2p+1,2q)&\aas\su(p,q)&&p\ne q-1,q\\
    \yy&\so(n,1)&\so(n-2q,1)+\su(q)+\ff&\ff\subset \uf(1)&2\le q\le\frac n2\\
    \yy&\so(n,1)&\so(n-4q,1)+\sp(q)+\ff&\ff\subset \sp(1)&2\le q\le\frac n4\\
    \yy&\so(n,1)&\so(n-16,1)+\spin(9)&&n\ge16\\
    \yy&\so(n,q)&\so(n-7,q)+\sG_2&&n\ge7, q=1,2\\
    \yy&\so(n,q)&\so(n-8,q)+\spin(7)&&n\ge8, q=1,2,3\\
    \yy&\so(6,3)&\so(2,0)+\sG_2^1\\
    \yy&\so(7,4)&\so(3,0)+\spin(4,3)\\
    \yy&\so^*(2n)&\aas\so^*(2n-2)&&n\ge5\\
    \yy&\so^*(10)&\aas\spin(6,1)\text{ or }*\!\spin(5,2)\\
    \hline
    \yy&\sE_6^4&\sl(3,\HH)+\ff&\ff\subset \uf(1)\\
    \yy&\sE_7^2&\aas\sE_6^2\text{ or }\ \  \aas\sE_6^3\\
    \yy&\sF_4^2&\sp(2,1)+\ff&\ff\subset \uf(1)\\
    \hline
  \end{array}\]
  \smallskip
  \centerline{\rm\Tabelle{table real spherical}}  
\end{table}

\begin{remark}\label{remark intro}\

  1. We use Berger's notation for the exceptional real Lie
  algebras. See Section \ref{exc-not}.

  2. There is some overlap between \ref{abs-sph} and \ref{not
    abs-sph}, as it appeared more useful to include a couple of
  absolutely spherical cases in Table \ref{table real spherical}. This
  holds for case (1) which is absolutely spherical unless
  $p_1+q_1= p_2+q_2$.  Moreover, case (2) is absolutely spherical when
  $q=\frac n2$, case (8) is absolutely spherical when $p+q$ is odd,
  and case (9) is absolutely spherical if $q=\frac n2$ and
  $\ff=\uf(1)$.

  3. The tables contain redundancies for small values of the
  parameters. These are mostly resolved by restricting $\gf$ to
\[
\begin{array}{lllll}
  \su(p,q)&\sl(n,\HH)&\sp(p,q)&\so(p,q)&\so^*(2n)\\
  p+q\ge2&n\ge2&p+q\ge2&p+q\ge 7&n\ge5
\end{array}
\]
and $p\ge q\ge 1$.

4. In Table \ref{table real spherical} the real spherical subalgebras
which are of codimension one in an absolutely spherical subalgebra are
marked with an $*$ in front of $\hf$ (with the exception of (2) and
(9) with $\ff=0$ and $n=2q$).  See Lemma \ref{center removed}.
	
5.  For simple Lie algebras $\gf$ of split rank one the real spherical
pairs were previously described in \cite{Kimelfeld}, and a more
explicit classification was later given in \cite{KM}.
\end{remark}

\subsection{Method of proof}

Our starting point is the following theorem which we prove in Sections
\ref{max SL}--\ref{max ex}, by making use of Dynkin's classification
of the maximal subalgebras in a complex simple Lie algebra.

\begin{theorem}\label{maximals are real forms}
  Let $(\gf,\hf)$ be a real spherical pair for which $\gf$ is simple
  and non-compact, and $\hf$ is a maximal reductive subalgebra.  Then
  $(\gf,\hf)$ is absolutely spherical.
\end{theorem}
 
Using Krämer's list \cite{Kr} we then also obtain the following lemma.

\begin{lemma}\label{most maximals are symmetric}
  Let $\gf$ be a non-compact simple real Lie algebra without complex
  structure and $\hf\subsetneq\gf$ be a maximal reductive subalgebra
  which is spherical. Then either $\hf$ is a symmetric subalgebra of
  $\gf$ or a real form of $\sl(3,\C)\subset \sG_2^\C$ or
  $\sG_2^\C\subset \so(7,\C)$.
\end{lemma}

In order to complete the classification we use the following
criterion, see Proposition \ref{tower-factor} and Corollary \ref{cor
  tower-factor}: If $(\gf, \hf)$ is real spherical with $\hf$
reductive and algebraic there exists a parabolic subalgebra
$\qf\supset \pf$ and a Levi decomposition $\qf=\lf +\uf$, such that
for every reductive and algebraic subalgebra $\hf'\subset\hf$, which
is also spherical in $\gf$, one has
\begin{equation}\label{fac} \hf = \hf' + (\lf \cap \hf).\end{equation}
In other words \eqref{fac} provides a factorization in the sense of
Onishchik. It is not too hard to determine all $\lf\cap \hf$ for
maximal $\hf$ (see Tables \ref{lcaph1} and \ref{lcaph2}). This allows
us to conclude the classification by means of Onishchik's list
\cite{Oni62} of factorizations of complex simple Lie algebras (see
Proposition \ref{Oni2}).

\subsection{Motivation}
This paper serves as the starting point for a follow up second part
which classifies all real spherical reductive subalgebras of
semisimple Lie algebras (see \cite{classII}).  With these
classifications one obtains an invaluable source of examples of real
spherical pairs.

Our main motivation for studying these pairs is that they provide a
class of homogeneous spaces $Z=G/H$, which appears to be natural for
the purpose of developing harmonic analysis.  Here $G$ is a reductive
Lie group and $H$ a closed subgroup.  The class includes the reductive
group $G$ itself, when considered as a homogeneous space for the
two-sided action.  In this case the establishment of harmonic analysis
is the fundamental achievement of Harish-Chandra \cite{HC-works}. More
generally a theory of harmonic analysis has been developed for
symmetric spaces $Z=G/H$ (see \cite{Delorme} and \cite{vdBS}). A
common geometric property of these spaces is that the minimal
parabolic subgroups of $G$ have open orbits on $Z$, a feature which
plays an important role in the cited works.  This property of the pair
$(G,H)$ is equivalent that the pair of their Lie algebras is real
spherical.  Recent developments reveal that a further generalization
of harmonic analysis to real spherical spaces is feasible, see
\cite{KS2}, \cite{KKSS}, the overview article \cite{KS3}, and
\cite{KKS2}.

\bigskip\noindent{\it Acknowledgment}: It is our pleasure to thank the two
referees for plenty of useful suggestions.  They resulted in a
significant improvement of the initially submitted manuscript.

\section{Generalities}

\subsection{Real spherical pairs}
In the sequel $\gf$ will always refer to a real reductive Lie algebra
and $\hf\subset\gf$ will be an algebraic subalgebra. The Lie algebra
$\hf$ is called {\it real spherical} provided there exists a minimal
parabolic subalgebra $\pf$ such that
\[ \gf = \hf +\pf\, .\] The pair $(\gf, \hf)$ is then referred to as a
{\it real spherical pair}.

Let $\theta$ be a Cartan involution of $\gf$, and let $\gf=\kf+\sf$
denote the corresponding Cartan decomposition.  Given a minimal
parabolic subalgebra $\pf$ we select a maximal abelian subspace $\af$
of $\sf$, which is contained in $\pf$, and write $\mf$ for the
centralizer of $\af$ in $\kf$. Then $\pf=\mf+\af+\nf$, where $\nf$ is
the unipotent radical of $\pf$.  Moreover $\dim(\gf/\pf)=\dim\nf$, and
hence this gives us the {\it dimension bound} for a real spherical
subalgebra $\hf\subset\gf$:
\begin{equation} \label{db} \dim \hf \geq \dim \nf =
  \dim(\gf/\kf)-\rank_\R \gf\, . \end{equation} We note that
$\dim(\gf/\kf)$ and $\rank_\R \gf$ are both listed in Table V of
\cite{Helgason}*{Ch.~X, p.~518}.  Further we record the obvious but
nevertheless sometimes useful {\it rank inequality}
\begin{equation} \label{rankcondition} \rank_\R \gf \geq \rank_\R \hf
  \, .\end{equation}

A pair $(\gf,\hf)$ of a complex Lie algebra and a complex subalgebra
is called {\it complex spherical} or just {\it spherical} if it is
real spherical when regarded as a pair of real Lie algebras.  Note
that in this case the minimal parabolic subalgebras of $\gf$ are
precisely the Borel subalgebras.

Given a pair $(\gf_\C,\hf_\C)$ of a complex Lie algebra and a
subalgebra, a {\it real form} of it is a pair $(\gf,\hf)$ of a real
Lie algebra and a subalgebra such that $\gf$ and $\hf$ are real forms
of $\gf_\C$ and $\hf_\C$, respectively.  We recall from the
introduction that the real form $(\gf,\hf)$ is called {\it absolutely
  spherical} when $(\gf_\C, \hf_\C)$ is spherical.  The following is
easily observed (see \cite{KKS2}*{Lemma 2.1}).

\begin{lemma}\label{real form}  
  All absolutely spherical pairs $(\gf,\hf)$ are real spherical.
\end{lemma}

We recall also that a pair $(\gf,\hf)$ is called {\it symmetric} in
case there exists an involution of $\gf$ for which $\hf$ is the set of
fixed elements, and that all such pairs are absolutely
spherical. Conversely we have the following result.

\begin{lemma} \label{real symmetric} Let $(\gf,\hf)$ be a real form of
  a complex symmetric pair $(\gf_\C,\hf_\C)$ with $\gf$ semisimple.
  Let $\sigma$ be the involution of $\gf_\C$ with fix point algebra
  $\hf_\C$.  Then $\sigma$ preserves $\gf$. In particular, $(\gf,\hf)$
  is symmetric.
\end{lemma}

\begin{proof} Let $\qf\subset \gf$ be the orthogonal complement of
  $\hf$ with respect to the Cartan-Killing form of $\gf$.  Then
  $\qf_\C$ is the orthogonal complement of $\hf_\C$ in $\gf_\C$ with
  respect to the Cartan-Killing form of $\gf_\C$. On the other hand
  $\qf_\C$ is the $-1$-eigenspace of $\sigma$. The assertion follows.
\end{proof}

Fix $\gf$ and let $G_\C$ be a linear complex algebraic group with Lie
algebra $\gf_\C=\gf\otimes_\R\C$.  We denote by $G$ the connected Lie
subgroup of $G_\C$ with Lie algebra $\gf$.  For any Lie subalgebra
$\lf\subset\gf$ we denote by the corresponding upper case Latin letter
$L\subset G$ the associated connected Lie subgroup, unless it is
indicated otherwise.

Let $P\subset G$ be a minimal parabolic subgroup. Then $Z:=G/H$ is
called a real spherical space provided that $(\gf, \hf)$ is real
spherical, which means that there is an open $P$-orbit on~$Z$.  In the
sequel we write $P=MAN$ for the decomposition of $P$ which corresponds
to the previously introduced decomposition $\pf=\mf+\af+\nf$ of its
Lie algebra, where the connected groups $A$ and $N$ are defined
through the convention above, and the possibly non-connected group $M$
is defined as the centralizer of $\af$ in $K$,

\subsection{Notation for classical and exceptional
  groups}\label{exc-not}

If $\gf_\C$ is classical, then $G_\C$ will be the corresponding
classical group, i.e.~$G_\C= \SL(n,\C), \SO(n,\C), \Sp(n, \C)$.  To
avoid confusion let us stress that we use the notation $\Sp(n,\R)$,
$\Sp(n,\C)$ to indicate that the underlying classical vector space is
$\R^{2n}$, $\C^{2n}$.  Further $\Sp(n)$ denotes the compact real form
of $\Sp(n,\C)$ and likewise the underlying vector space for $\Sp(p,q)$
is $\C^{2p+2q}$.

By $\SL(n,\HH)\subset\SL(2n,\C)$ and $\SO^*(2n)\subset\SO(2n,\C)$ we
denote the subgroups of elements $g$ which satisfy
\[gJ=J\bar g,\qquad J=\begin{pmatrix} 0&I_n\\-I_n&0\end{pmatrix}\]
where $I_n$ denotes the identity matrix of size $n$.  Another standard
notation for $\SL(n,\HH)$ is $\SU^*(2n)$.

We denote by ${\mathrm O}(p,q)$ the indefinite orthogonal group on
$\R^{p+q}$.  The identity component of ${\mathrm O}(p,q)$ is denoted
by $\SO_0(p,q)$.

For exceptional Lie algebras we use the notation of Berger,
\cite{Berger}*{p.~117}, and write $\sE_6^\C, \sE_7^\C$ etc.~for the
complex simple Lie algebras of type $E_6, E_7$ etc., and
$\sE_6, \sE_7$ etc.~for the corresponding compact real forms. For the
non-compact real forms we write
\begin{eqnarray*}
  \sE_6^1,\sE_6^2,\sE_6^3,\sE_6^4\qquad&\hbox{for}&\qquad\mathrm{E\,I,E\,II,E\,III,E\,IV}\\
  \sE_7^1,\sE_7^2,\sE_7^3\qquad&\hbox{for}&\qquad\mathrm{E\,V,E\,VI,E\,VII}\\
  \sE_8^1,\sE_8^2\qquad&\hbox{for}&\qquad\mathrm{E\,VIII,E\,IX}\\
  \sF_4^1,\sF_4^2\qquad&\hbox{for}&\qquad\mathrm{F\,I,F\,II}
\end{eqnarray*}
and finally $\sG_2^1$ for $\mathrm G$, the unique non-compact real
form of $\sG_2^\C$.  By slight abuse of notation we denote the simply
connected Lie groups with exceptional Lie algebras by the same
symbols.

\subsection{Factorizations of reductive groups}

Let $\hf$ be a reductive Lie algebra.  Then a triple
$(\hf, \hf_1, \hf_2)$ is called a {\it factorization of $\hf$} if
$\hf_1$ and $\hf_2$ are reductive subalgebras of $\hf$ and
\begin{equation}\label{factor1} \hf= \hf_1 +\hf_2\, .\end{equation}
It is called trivial if one of the factors equals $\hf$.  Recall that
a reductive subalgebra of $\hf$ is a subalgebra for which $\ad_\hf$ is
completely reducible.

Likewise if $H$ is a connected reductive group and $H_1$ and $H_2$ are
connected reductive subgroups of $H$, then we call $(H, H_1, H_2)$ a
{\it factorization of $H$} provided that
\begin{equation} \label{factor2} H=H_1 H_2\, .\end{equation}

\begin{proposition}[Onishchik \cite{Oni69}] \label{Oni1} Let
  $H$ be a connected reductive group and $H_1, H_2$ reductive
  subgroups of $H$. Then the following are equivalent:
  \begin{enumerate}
  \item\label{algebra factorization} $(\hf, \hf_1, \hf_2)$ is a
    factorization of $\hf$.
  \item\label{group factorization} $(H, H_1, H_2)$ is a factorization
    of $H$.
  \item\label{corOni1} $H_1 x H_2 \subset H$ is open for some
    $x\in H$.
  \end{enumerate}
\end{proposition}

\begin{proof} We refer to \cite{Ak}*{Prop.~4.4}, for the equivalence of
  \ref{algebra factorization} and \ref{group factorization}. It is
  obvious that (\ref{factor2}) implies $H_1xH_2=H$ for all $x$, and
  hence in particular \ref{group factorization} implies
  \ref{corOni1}.

  Assume \ref{corOni1}, then \ref{algebra factorization} is valid
  for the pair of $\hf_1$ and $\Ad(x)\hf_2$. Hence \ref{group
    factorization} holds for the pair of $H_1$ and $xH_2x^{-1}$. This
  implies $H=H_1xH_2$ and thus $H_1\times H_2$ acts transitively on
  $H$, that is, \ref{group factorization} holds for $H_1,H_2$.
\end{proof}

As a consequence we obtain the following result.  Here we call a
subalgebra of $\gf$ compact if it generates a compact subgroup in the
adjoint group of $\gf$.

\begin{lemma} \label{factor compact} Let $\gf$ be a semisimple Lie
  algebra without compact ideals. Then every factorization of $\gf$ by
  a reductive and a compact subalgebra is trivial.
\end{lemma}

\begin{proof} Let $\gf=\hf_1+\hf_2$ be as assumed.  Since $\hf_1$ is a
  reductive subalgebra there exists a Cartan involution which leaves
  it invariant. Let $\gf=\kf+\sf$ denote the corresponding Cartan
  decomposition, and note that $\kf=[\sf,\sf]$ since $\gf$ has
  no compact ideals.  Without loss of generality we may assume that
  $\hf_2$ is a maximal compact subalgebra, hence conjugate to $\kf$.
  It then follows from Proposition \ref{Oni1} that
  $\gf=\hf_1+\kf$. Hence $\sf\subset\hf_1$. Then
  $\gf=[\sf,\sf]+\sf= \hf_1$ and the factorization is
  trivial.
\end{proof}

Factorizations of simple complex Lie algebras were classified in
\cite{Oni62} as follows.

\begin{proposition} \label{Oni2} {\rm (Onishchik)} Let $\gf$ be a
  complex simple Lie algebra and let $\gf=\hf_1+\hf_2$ where $\hf_1$
  and $\hf_2$ are proper reductive complex subalgebras of $\gf$. Then,
  up to interchanging $\hf_1$ and $\hf_2$, the triple
  $(\gf,\hf_1,\hf_2)$ is isomorphic to a triple in Table
  \ref{Onishchik}, where line by line, $\zf\subset\C$ and
  $\ff\subset\sp(1,\C)$.
\end{proposition}
	
\begin{table}[ht]
  \setcounter{class}{0}
  \[
    \begin{array}{r l l l l l}
      &\gf&\hf_1&\hf_2&\hf_1\cap\hf_2\\
      \hline
      \yy&\sl(2n,\C)&\sl(2n-1,\C)+\zf&\sp(n,\C)
                      &\sp(n-1,\C)+\zf&n\ge2\\
      \yy&\so(2n,\C)&\so(2n-1,\C)&\sl(n,\C)+\zf
                      &\sl(n-1,\C)+\zf&n\ge4\\
      \yy&\so(4n,\C)&\so(4n-1,\C)&\sp(n,\C)+\ff
                      &\sp(n-1,\C)+\ff&n\ge2\\
      \yy&\so(7,\C)&\so(5,\C)+\zf&\sG_2^\C
                      &\sl(2,\C)+\zf\\
      \yy&\so(7,\C)&\so(6,\C)&\sG_2^\C
                      &\sl(3,\C)\\
      \yy&\so(8,\C)&\so(5,\C)+\ff&\spin(7,\C)
                      &\sl(2,\C)+\ff\\
      \yy&\so(8,\C)&\so(6,\C)+\zf&\spin(7,\C)
                      &\sl(3,\C)+\zf\\
      \yy&\so(8,\C)&\so(7,\C)&\spin(7,\C)
                      &\sG_2^\C\\
      \yy&\so(8,\C)&\spin(7,\C)_+&\spin(7,\C)_-
                      &\sG_2^\C\\
      \yy&\so(16,\C)&\so(15,\C)&\spin(9,\C)
                      &\spin(7,\C)\\
    \end{array}
  \]
  \centerline{\rm\Tabelle{Onishchik}}
\end{table}

\begin{remark}\label{embeddings remark} \
  \begin{enumerate}
  \item The spin representation embeds $\spin(7,\C)$ into $\so(8,\C)$
    and there are two conjugacy classes of this subalgebra. In
    Table~\ref{Onishchik} (9) the subscripts indicate that this
    factorization involves both conjugacy classes.
	
  \item In all cases $\hf_1$ is given up to conjugation in $\gf$.
    Once $\hf_1$ is fixed, there is only one $\Ad(H_1)$-conjugacy
    class of $\hf_2$ in $\gf$ for which the factorization is valid,
    except where $\hf_2=\spin(7,\C)$ is indicated without subscript.
    In those cases there are exactly two such conjugacy classes,
    provided by $\spin(7, \C)_\pm$.
	
  \item Observe that symplectic or exceptional Lie algebras do not
    admit factorizations.
  \end{enumerate}
\end{remark}

\subsection{Towers of spherical subgroups}\label{Towers}

Let $Z=G/H$ be a real spherical space and $P\subset G$ a minimal
parabolic subgroup such that $PH$ is open in $G$.  Then the local
structure theorem of \cite{KKS} asserts that there is a parabolic
subgroup $Q\supset P$ with Levi decomposition $Q= L \ltimes U$ such
that:

\begin{enumerate}
\item\label{LST1} $PH=QH$.
\item\label{LST2} $Q\cap H=L\cap H$.
\item\label{LST3} $L_{\mathrm n}\subset L\cap H$.
\end{enumerate}
Here $L_{\mathrm n}\subset L$ is the normal subgroup with Lie algebra
$\lf_{\mathrm n}$, the sum of all non-compact simple ideals of $\lf$.
We refer to $Q$ and its Levi part $L$ as being adapted to $Z$ and $P$,
taking it for granted that $PH$ is open.

\begin{remark} \label{remark cplx adapted}In the special case where
  $Z$ is complex spherical note that
  $\lf_{\rm n}=[\lf,\lf]$.\end{remark}

\begin{lemma}\label{min psgp LcapH} 
  Let $H\subset G$ be reductive and real spherical, and let $Q=LU$ be
  adapted to $G/H$ and $P$.  Then $L\cap H$ is reductive and contains
  $P\cap H$ as a minimal parabolic subgroup.\end{lemma}

\begin{proof} It follows from \ref{LST3} above that $\lf_n$ is a
  semisimple ideal in $\lf\cap\hf$. As the quotient consists of
  abelian or compact factors, $\lf\cap\hf$ is reductive.  Since
  $P\cap L$ is a minimal parabolic subgroup in $L$, it also follows
  from \ref{LST3} that $P\cap L\cap H$ is a minimal parabolic
  subgroup in $L\cap H$.  Since $P\subset Q$ it follows from
  \ref{LST2} that $P\cap H=P\cap L\cap H$.
\end{proof}

\begin{proposition} \label{tower-factor} Let $H'\subset H\subset G$ be
  subgroups such that $H$ is reductive and $G/H$ is real spherical,
  and let $Q=LU$ be adapted to $G/H$ and $P$.  Then $G/H'$ is real
  spherical if and only if $H/H'$ is real spherical for the action of
  $L\cap H$, that is, it admits an open orbit for the minimal
  parabolic subgroup $P\cap H$ (cf.~Lemma \ref{min psgp LcapH}).
\end{proposition}

\begin{proof} Assume $G/H'$ is real spherical. Then, by density of the
  union of the open orbits, for some $x\in G$ the set $PxH'$ is open
  in $G$ and intersects non-trivially with the open set $PH$.  It
  follows that $PyH'$ is open in $G$ for some $y\in H$. Then the
  intersection $(PyH')\cap H=(P\cap H)yH'$ is open in $H$.

  Conversely, it is clear that if $(PyH')\cap H$ is open in $H$ for
  some $y\in H$, then $PyH'$ is open in $PH$ and hence in $G$.
\end{proof}

\begin{corollary} \label{cor tower-factor} Let $H'\subset H\subset G$
  be reductive subgroups and let $Q$ be as above. If $Z'=G/H'$ is real
  spherical then $(H, H', L\cap H)$ is a factorization of $H$, that
  is,
  \begin{equation} \label{product with L} H= H' (L\cap H)\,
    .\end{equation}
  Conversely, if $Q=P$ then \eqref{product with L} implies that $Z'$
  is real spherical.
\end{corollary}

\begin{proof} It follows from Proposition \ref{tower-factor} that
  $(L\cap H)xH'$ is open in $H$ for some $x\in H$. Then \eqref{product
    with L} follows from Proposition \ref{Oni1}.  Conversely,
  \eqref{product with L} implies that $H'Q=HQ$, and hence $Z'$ is
  spherical if $Q=P$.
\end{proof}

We recall from \cite{KK}*{Prop. 9.1} the following consistency
relation of adapted parabolics.

\begin{lemma} \label{adapQ}Let $Z=G/H$ be a real form of a complex
  spherical space $Z_\C = G_\C/H_\C$.  Let $P\subset G$ be a minimal
  parabolic subgroup and $B_\C \subset G_\C$ a Borel subgroup such
  that $B_\C \subset P_\C$ and $B_\C H_\C \subset G_\C$ open. Let
  ${\mathcal Q}_\C\supset B_\C$ be the $Z_\C$-adapted parabolic
  subgroup of $G_\C$ and $Q\supset P$ the $Z$-adapted parabolic
  subgroup of $G$. Then
  \[ Q_\C = {\mathcal Q}_\C M_\C\, .\]
\end{lemma}

\begin{remark}\label{remark curlyQ} Suppose that $(\gf, \hf)$ is
  absolutely spherical with $\hf$ self-normalizing.  Let $H_\C$ be the
  normalizer of $\hf_\C$ in $G_\C$.  Note that $H_\C$ is a
  self-normalizing spherical subgroup of $G_\C$.  In view of
  \cite{Kn}*{Cor.~7.2} this implies that $Z_\C=G_\C/H_\C$ admits a
  wonderful compactification and as such is endowed with a Luna
  diagram, see \cite{BP}.

  The Luna diagram consists of the Dynkin diagram of $\gf_\C$ with
  additional information.  In particular the roots corresponding to
  the adapted Levi ${\mathcal L}_\C\subset{\mathcal Q}_\C$ are the
  uncircled elements in the Luna diagram where ``uncircled'' means no
  circle around, above, or below a vertex in the underlying Dynkin
  diagram.  Combining this information with the Satake diagram of
  $\gf$ then gives us the structure of $L$ via Lemma \ref{adapQ}.

  In view of Remark \ref{remark cplx adapted} we have
  \begin{equation} \label{cL for P=Q} [\operatorname{Lie}({\mathcal
      L}_\C), \operatorname{Lie}({\mathcal L}_\C)]\subset
    \hf_\C\end{equation} and in particular
  \begin{equation} \label{L for P=Q} [\lf,\lf]\subset \hf \qquad
    \hbox{in case $Q_\C={\mathcal Q}_\C$}\, . \end{equation}
\end{remark}

\subsection{The case where $\gf$ is a quasi-split real form of
  $\gf_\C$}
Recall that $\gf$ is called quasi-split if the complexification
$\pf_\C$ of $\pf$ is a Borel subalgebra of $\gf_\C$. An equivalent way
of saying this is that $\mf$ is abelian.  The following is clear.

\begin{lemma} \label{lemma quasi-split} Let $(\gf,\hf)$ be a real form
  of $(\gf_\C,\hf_\C)$ and assume that $\gf$ is quasi-split. Then
  $(\gf, \hf)$ is real spherical if and only if $(\gf_\C, \hf_\C)$ is
  spherical.
\end{lemma}

\subsection{Two technical Lemmas}
We conclude this section with two lemmas which are repeatedly used in
the classification later on.  The first is variant of Schur's Lemma.

\begin{lemma} \label{signature lemma}Let $V$ be a finite dimensional
  complex vector space endowed with a non-degenerate Hermitian form
  $b$.  Further let $G\subset \GL(V)$ be a subgroup which acts
  irreducibly on $V$ and leaves $b$ invariant.  Then any other
  $G$-invariant Hermitian form $b'$ on $V$ is a real multiple of
  $b$. In particular if $b'\neq 0$, then $b$ and $b'$ have the same
  signature $(p,q)$ up to order.
\end{lemma}

\begin{proof} Since $b$ is non-degenerate we find a unique
  $T\in \operatorname{End}_\C(V)$ such that $b'(v,w) = b(Tv,w)$ for
  all $v, w\in V$. The $G$-invariance of both $b$ and $b'$ and the
  uniqueness of $T$ then implies that $gTg^{-1}=T$ for all $g\in
  G$. Since $G$ acts irreducibly on $V$, Schur's Lemma implies that
  $T=\lambda\cdot \operatorname{id}_V$ for some $\lambda\in \C$.
  Since both $b$ and $b'$ are Hermitian the scalar $\lambda$ needs to
  be real.
\end{proof}

\begin{lemma} \label{independent} Let $X$ be a real algebraic variety
  acted upon by a real algebraic group $H$.  Further let
  $ f_1,\dots,f_k$ be $H$-invariant rational functions on $X$.  Let
  $U\subset X$ be their common set of definition. Consider
  \[ F: U \to\R^k, \ \ x\mapsto (f_1(x), \ldots, f_k(x))\] and assume
  that
  \[V:=\{ x\in U \mid \rank dF(x) \geq k\}\neq \emptyset\, .\] Then
  \[ \max_{x\in X} \dim_\R Hx \leq \dim X - k\, .\]
\end{lemma}

\begin{proof} Note that $V$ is by assumption Zariski open in
  $X$. Hence generic $H$-orbits of maximal dimension meet $V$.  Since
  level sets in $V$ under $F$ have codimension $k$, the assertion
  follows.
\end{proof}

Functions $f_1, \ldots, f_k$ as above which meet the requirement
$V\neq \emptyset$ will in the sequel be called {\it independent}.

\section{The Dynkin scheme of maximal reductive subgroups of classical
  groups}

Let $G_\C$ be a complex classical group and let $V$ be the standard
representation space attached to $G_\C$, i.e.~$V=\C^n$ for
$G_\C=\SL(n,\C)$ or $\SO(n,\C)$, and $V=\C^{2n}$ for $G_\C=\Sp(n,\C)$.
According to Dynkin~\cite{D}, there are three possible types of a
connected maximal complex reductive subgroup $H_\C$ of $G_\C$.

\subsection{Type I: The action of $H_\C$ on $V$ is reducible} Up to
conjugation $H_\C$ is one of the following subgroups, which are all
symmetric:

\subsubsection{$G_\C=\SL(n,\C)$}
Here $H_\C= S(\GL(n_1,\C)\times \GL(n_2,\C))$, $n=n_1+n_2$, $n_i>0$.

\subsubsection{$G_\C=\SO(n,\C)$} Either
$H_\C= \SO(n_1,\C)\times \SO(n_2,\C)$ with $n=n_1+n_2$, $n_i>0$ or $n$
is even and $H_\C= \GL(n/2,\C)$. In the first case, the defining
bilinear form on $G_\C$ restricts non-trivially to the factors
$\C^n=V= V_1+ V_2=\C^{n_1} \oplus \C^{n_2}$. In the second case
$V=V_1\oplus V_1^*$ for $V_1$ the standard representation of
$\GL(n/2,\C)$ and both factors $V_1$ and $V_1^*$ are isotropic.

\subsubsection{$G_\C=\Sp(n,\C)$} Here
$H_\C= \Sp(n_1,\C)\times \Sp(n_2,\C)$ with $n=n_1+n_2$, $n_i>0$, or
$H_\C=\GL(n,\C)$. In the first case, the defining bilinear form on
$G_\C$ restricts non-trivially to the factors
$\C^{2n}=V= V_1+ V_2=\C^{2n_1} \oplus \C^{2n_2}$. In the second case
$V=V_1\oplus V_1^*$ for $V_1$ the standard representation of
$\GL(n,\C)$ and both factors $V_1$ and $V_1^*$ are Lagrangian.

\subsection{Type II: The action of $H_\C$ on $V$ is irreducible, but
  $\hf_\C$ is not simple}

\subsubsection{$G_\C=\SL(n,\C)$}
Here $H_\C=\SL(r,\C)\otimes\SL(s,\C)$ and $\C^n = \C^r \otimes \C^s$
with $rs=n$ and $2\leq r\leq s$.

\subsubsection{$G_\C=\SO(n,\C)$}
Here there are two possibilities. The first is
$H_\C=\SO(r,\C)\otimes\SO(s,\C)$ acting on $\C^n=\C^r\otimes\C^s$ with
$n=rs$, $3\leq r\leq s$, and $r,s\neq 4$.  The second case is
$H_\C=\Sp(r,\C) \otimes \Sp(s,\C)$ acting on
$\C^n=\C^{2r}\otimes \C^{2s}$ with $n=4rs$ and $1\leq r\leq s$.

\subsubsection{$G_\C=\Sp(n,\C)$}
Here $H_\C=\Sp(r,\C)\otimes \SO(s,\C)$ and
$\C^{2n}=\C^{2r}\otimes\C^s$ with $n=rs$ and $r\ge 1$, $s\geq 3$.
Moreover it is requested that $s\neq 4$ unless if $r=1$.

\subsection{Type III: The action of $H_\C$ on $V$ is irreducible and
  $\hf_\C$ is simple.} \label{typeIII} For this type the different
cases are listed in \cite{D}*{Thm.~1.5}. However, we do not need this
list.

\subsection{Dynkin types in $G$}

Let $H\subset G$ be a maximal connected reductive subgroup. Note that
this implies that $\hf$ is a maximal reductive subalgebra in $\gf$.
To begin with we recall the following result:

\begin{proposition}\label{prop Komrakov}  {\rm (Komrakov \cite{Kom1}, \cite{Kom2})} Let $\gf$ be a real simple Lie algebra 
  and $\hf$ a maximal reductive subalgebra.  If $\hf_\C$ is not
  maximal reductive in $\gf_\C$, then the pair $(\gf,\hf)$ appears in
  the following list:
  \begin{enumerate}
  \item \label{Kom1}$ (\sp(4n,\R),\, \so(1,3)\oplus \sp(n,\R))$,
    $n\geq 2$
  \item
    \label{Kom2}$ (\sp(p+3q, 3p+q),\, \so(1,3) \oplus
    \sp(p,q))$, $p+q\geq 2$
  \item \label{Kom3}$ (\so^*(8n),\, \so(1,3)\oplus \so^*(2n))$,
    $n\geq 2$
  \item \label{Kom4}$ (\so (p+3q, q+3p),\, \so(1,3)\oplus \so(p,q))$,
    $p+q\geq 3$
  \item \label{Kom5}$ (\so (6,10),\, \so(1,3) \oplus \so(1,3))$
  \item \label{Kom6}$( \so(165, 330),\, \so(1,11))$
  \item \label{Kom7}$( \so(234, 261),\, \so(3,9))$
  \item \label{Kom8}$( \sE_8^2,\, \sG_2^\C\oplus \su(2))$
  \item \label{Kom9}$( \sE_8^1,\, \sG_2^\C\oplus \su(1,1))$
  \end{enumerate}
\end{proposition}

\begin{remark} The particular embeddings of $\hf$ into $\gf$ in
  Proposition \ref{prop Komrakov} are described in \cite{Kom1}.  For
  this article the particular embeddings are not needed as only
  $\dim \hf$ enters in the proof of Corollary \ref{Komrakov cor}
  below.
\end{remark}

\begin{corollary}\label{Komrakov cor} 
  Let $\gf$ be a real simple Lie algebra and $\hf$ a real spherical
  maximal reductive subalgebra. Then $\hf_\C$ is maximal reductive in
  $\gf_\C$.
\end{corollary}

\begin{proof} Recall the dimension bound $\dim \hf\geq \dim \nf$ from
  \eqref{db}.  Now for \ref{Kom1} we note that $\dim \nf = (4n)^2$
  whereas $\dim \hf= 6+ 2n^2 +n$. For \ref{Kom2} we use the
  dimension bound \eqref{dimsp}, for \ref{Kom3} the dimension bound
  \eqref{dimsoev}, for \ref{Kom4}--\ref{Kom7} the dimension bound
  \eqref{dimso}, and finally we exclude \ref{Kom8} and \ref{Kom9}
  via the dimension bound \eqref{dimE_8^2}:
  $\dim \nf(\sE_8^1) \geq \dim \nf(\sE_8^2)=108$.
\end{proof}

\begin{definition} Let $G$ be a real classical group. We say that
  \emph{a maximal connected reductive subgroup $H$ is of type I, II,
    or III}, provided $H_\C$ is maximal reductive and of that type in
  $G_\C$.
\end{definition}

\begin{remark} \label{rem:typeII-complexsubgroup} Suppose that $V$ is
  the complexification of a real vector space $V_\R$ and that
  $H\subset G\subset\GL(V_\R)$ with $H_\C\subset G_\C$ maximal
  reductive. Suppose that there exists a complex structure $J_\R$ on
  $V_\R$ such that $H$ is a complex subgroup of $\GL(V_\R, J_\R)$.
  Then $H$ is of type I.  Indeed $H_\C \simeq H \times \oline H$ and
  the action of $H_\C$ on $V\simeq (V_\R, J_\R)\oplus (V_\R, -J_\R)$
  is reducible.
\end{remark}

\subsection{Bilinear forms on prehomogeneous vector spaces}

Let $G$ be an algebraic group over $\C$ and $\rho$ be a
finite-dimensional representation of $G$ on a complex vector space
$V$. The triplet $(G,\rho,V)$ is called a {\it prehomogeneous vector
  space}, if $G \times \GL(1,\C)$ has a Zariski open orbit in $V$.

Let now $V$ be an irreducible representation of a reductive group $G$
with center at most one-dimensional, and let $G'$ denote the
semisimple part of $G$.  A necessary condition for $(G,\rho,V)$ to be
prehomogeneous is that $G'$ satisfies
\begin{equation}\label{equ:prehomodimension}
  \dim(G') \geq \dim(V) - 1.
\end{equation}

Two triplets $(G_1,\rho_1,V_1)$ and $(G_2,\rho_2,V_2)$ are said to be
{\it equivalent} if there is a linear isomorphism
$\psi: V_1 \rightarrow V_2$ such that
$\hat{\psi}(\rho_1(G_1)) = \rho_2(G_2)$ under the induced map
$\hat{\psi}: \GL(V_1) \rightarrow \GL(V_2)$.

With respect to this notion, $(G,\rho,V)$ and $(H,\rho\circ\tau, V)$
are equivalent whenever $\tau: H \rightarrow G$ is a surjective
homomorphism. In particular, $(G,\rho,V)$ is always equivalent to
$(G,\rho^*,V^*)$ where $\rho^*$ is dual to $\rho$.

\begin{proposition}\label{prop:prehomogeneous}
  Let $(\rho, V)$ be an irreducible representation of a simple group
  $G$. The triplet $(G,\rho,V)$ satisfies \eqref{equ:prehomodimension}
  if and only if it is equivalent to a triplet listed in
  Table~\ref{tab:prehomogeneous} and it gives rise to a prehomogeneous
  vector space if and only if it is marked in the column `preh'.
\end{proposition}

The table identifies the representation $\rho$ by its highest weight
(expanded in fundamental weights using the Bourbaki numbering
\cite{Bourbaki}*{Ch.\ 6, Planches I--X}) and dimension.

\begin{proof} The cases of Table \ref{tab:prehomogeneous} were
  determined in \cite{AVE} and the fourth column follows from Theorem
  54 in \cite{SK}.
\end{proof}

\def\spinthirteen{18}\def\Esix{22}\def\Eseven{23}\def\Ffour{21}
\begin{table}[ht]
  \small 
  \begin{tabular}{rlllll}
    \hline			
    & $G$                         & $\rho$      & $\dim(V)$          & preh.     & inv. form \\ \hline \hline
    1. & $G$ simple                  & adjoint     & $\dim G$           &           & 1         \\
    2. & $\SL(n,\C)$, $n \geq 3$     & $\omega_1$  & $n$                & \checkmark& 0	        \\
    3. & $\SL(n,\C)$, $n \geq 3$     & $2\omega_1$ & $\frac{1}{2}n(n+1)$& \checkmark& 0		    \\
    4. & $\SL(n,\C)$, $n \geq 5$     & $\omega_2$  & $\frac{1}{2}n(n-1)$& \checkmark& 0		\\
    5. & $\Sp(n,\C)$, $n \geq 1$     & $\omega_1$  & $2n$               & \checkmark& 2 \\
    6. & $\Sp(n,\C)$, $n \geq 3$     & $\omega_2$  & $(n-1)(2n+1)$      &           & 1      \\
    7. & $\SO(n,\C)$, $n\geq 3$, $n\neq 4$ & $\omega_1$  & $n$          & \checkmark& 1   \\ \hline 
    8. & $\SL(2,\C)$                 & $3\omega_1$ & $4$                & \checkmark& 2         \\
    9. & $\SL(6,\C)$                 & $\omega_3$  & $20$               & \checkmark& 2         \\
    10. & $\SL(7,\C)$                & $\omega_3$  & $35$               & \checkmark& 0          \\
    11. & $\SL(8,\C)$                & $\omega_3$  & $56$               & \checkmark& 0           \\
    12. & $\Sp(3,\C)$                & $\omega_3$  & $14$               & \checkmark& 2           \\
    13. & $\Spin(7,\C)$              & spin        & $8$                & \checkmark& 1         \\
    14. & $\Spin(9,\C)$              & spin        & $16$               & \checkmark& 1            \\
    15. & $\Spin(10,\C)$             & half spin   & $16$               & \checkmark& 0             \\
    16. & $\Spin(11,\C)$             & spin        & $32$               & \checkmark& 2             \\
    17. & $\Spin(12,\C)$             & half spin   & $32$               & \checkmark& 2              \\
    \spinthirteen. & $\Spin(13,\C)$  & spin        & $64$               &           & 2        		\\  
    19. & $\Spin(14,\C)$             & half spin   & $64$               & \checkmark& 0             \\
    20. & $\sG_2^\C$                 & $\omega_1$  & $7$                & \checkmark& 1        \\
    \Ffour. & $\sF_4^\C$             & $\omega_4$  & $26$               &           & 1     	\\
    \Esix. & $\sE_6^\C$              & $\omega_1$  & $27$               & \checkmark& 0          \\
    \Eseven. & $\sE_7^\C$            & $\omega_7$  & $56$               & \checkmark& 2     \\	 \hline  
    \multicolumn{6}{c} {\Tabelle{tab:prehomogeneous}}\\
  \end{tabular}

\end{table}

Table~\ref{tab:prehomogeneous} divides into two parts: Each triplet
listed in the first part represents a series of vector spaces while a
triplet in the second part is only defined for a certain dimension.
We call a triplet $(G,\rho,V)$ {\it classical} or {\it sporadic}
depending on whether it is equivalent to a triplet of the former or
the latter type.

The final column of the table is marked by $0$ if there is no
non-degenerate $G$-invariant bilinear form on $V$, and by $1$
(resp.~$2$) if there exists a non-degenerate symmetric (resp.~skew
symmetric) $G$-invariant bilinear form. Given the highest weight
$\omega$, this data is easily determined by means of \cite{Bourbaki}*{Ch.\ 8,
\S7.5, Prop.\ 12}.

\begin{remark}\label{remark on forms}
  Let $G$ be a simple classical group acting on $V$ as described in
  the beginning of this section, and let $H$ be a subgroup of type
  III. If $G$ is a real form of $\SO(n,\C)$, resp. $\Sp(n,\C)$, then
  $H_\C$ fixes a symmetric, resp. skew symmetric bilinear form on $V$.
  On the other hand, if $G$ is a real form of $\SL(n,\C)$ then $H$
  cannot be maximal if it fixes a nondegenerate bilinear form, unless
  $H_\C$ is conjugate to $\SO(n,\C)$ or (if $n$ is even) to
  $\Sp(\frac n2,\C)$.

  It will be a consequence of the dimension bound \eqref{db}, that in
  most cases a subgroup $H$ of type III comes from a triplet in
  Table~\ref{tab:prehomogeneous}.  Hence the provided information
  about invariant forms reduces the number of cases which must be
  considered for the classification of these subgroups.
\end{remark}

\section{Maximal reductive real spherical subgroups in case $G_\C=\SL(n,\C)$}\label{max SL}

We prove the statement in 
Theorem \ref{maximals are real forms} 
for $\gf_\C=\sl(n,\C)$ and $\hf_\C$ maximal reductive
(cf.~Corollary \ref{Komrakov cor}).

\subsection{The real forms}
It suffices to consider the non-split real forms $G=\SU(p,q)$ with
$p+q=n$ and $1\leq p\leq q$, and $G=\SL(m,\Hb)$ with $n=2m>2$.  For
these groups we obtain the following dimension bounds from the table
of \cite{Helgason} cited below \eqref{db}:
\begin{equation}\label{dimsu} \dim H \geq 2pq - p \qquad
  (G=\SU(p,q)),\end{equation}
\begin{equation}\label{dimsl} \dim H \geq 2 m^2 -2 m\qquad
  (G=\SL(m,\Hb)).\end{equation}

For later reference we record the matrix realizations of $G$ and $P$.
We begin with $G=\SU(p,q)$ which we consider as the invariance group
of the Hermitian form $(\cdot, \cdot)_{p,q}$ defined by the symmetric
matrix
\[J=\begin{pmatrix} 0 & I_p & 0\\ I_p & 0 & 0 \\ 0 & 0 &
    I_{q-p}\end{pmatrix}\, .\]

The Lie algebra is then given by
\[\su(p,q)=\left\{X=\begin{pmatrix} A & B & E \\ C & -A^*& F \\ -F^* &
      - E^* & D\end{pmatrix}\middle| \quad
    \begin{matrix}
      A, B, C\in \Mat_{p, p} (\C),\\
      E, F \in \Mat_{p, q-p}( \C), \\
      D \in \Mat_{q-p, q-p}(\C),\\
      B^*, C^*, D^*=-B, -C, -D,\\
      \tr(X)=0\end{matrix}\, \right\}\, .\] We choose the minimal
parabolic such that
\[\pf=\left\{\begin{pmatrix} A & B & E \\ 0 & -A^*& 0\\ 0 & - E^* &
      D\end{pmatrix}\in \su(p,q)\ \bigg|\, A \ \hbox{upper
      triangular}\right\}\] so that $P$ stabilizes the isotropic flag
$\la e_1\ra \subset \la e_1, e_2\ra \subset\ldots\subset \la e_1,
\ldots, e_p\ra $.  Moreover we record that
\[G/P\simeq\{ V_1\subset V_2 \subset \ldots \subset V_p\mid \dim_\C
  V_i = i, (V_p, V_p)_{p,q}=\{0\}\}\] is the variety of full isotropic
$p$-flags in $\C^{p+q}$.  We denote by
\[\Nc_{p,q}:= \{ [v]\in \Pb(\C^n)\mid (v,v)_{p,q}=0\}\]
the null-cone and note that there is a $G$-equivariant surjective map
$G/P \to \Nc_{p,q}$.  Moreover, if $P_{\rm max, \C}$ denotes the
maximal parabolic subgroup of $G_\C=\SL(n,\C)$ which stabilizes
$\la e_1\ra$, then $P_\C \subset P_{{\rm max}, \C}$ and thus we have a
$G_\C$-equivariant surjection
$G_\C/ P_\C \to G_\C/ P_{\rm max,\C}=\Pb(\C^n)$.

Thus we get:

\begin{lemma} \label{orbit bound SU}Let $G=\SU(p,q)$ and $H\subset G$
  be a real spherical subgroup. Then the following assertions hold:
  \begin{enumerate}
  \item There exists an $H$-orbit on $\C^n$ of real codimension at
    most 3.
  \item There exists an open $H_\C$-orbit on $\Pb(\C^n)$, i.e.~$\C^n$
    is a pre\-ho\-mo\-ge\-neous vector space for $H_\C$.
  \end{enumerate}
\end{lemma}

\begin{proof} The fact that $H$ has an open orbit on $G/P$ implies
  that there is an open $H$-orbit in $\Nc_{p,q}$, hence the first
  assertion.  Secondly, the fact that $H$ has an open orbit on $G/P$
  implies that $H_\C$ has an open orbit on $G_\C/ P_\C$ whence on
  $G/P_{\rm max, \C}$.
\end{proof}

For the group $G=\SL(m,\Hb)$ in $G_\C=\SL(2m,\C)$ we choose
$P\subset G$ the upper triangular matrices.  Then
\[G/P=\{ V_1\subset V_2 \subset \ldots\subset V_m = \Hb^m\mid \dim_\Hb
  V_i =i\}\] and in particular we obtain a $G$-equivariant surjection
$G/P\to \Pb(\Hb^m)$.  Hence we get:

\begin{lemma} \label{orbit bound SU*} Let $H\subset\SL(m,\Hb)$ be a
  real spherical subgroup.  Then $H$ has an open orbit on $\Pb(\Hb^m)$
  and an orbit on $\Hb^m=\C^{2m}$ of real codimension at most $4$.
\end{lemma}

\subsection{Exclusions of sphericity via the codimension
  bound}\label{exclusion orbit bound}

The criterion in Lemma \ref{orbit bound SU} (1) is quite useful to
show that many naturally occurring subgroups are not spherical. We
give an application in the lemma below.

\begin{lemma} \label{supq}Let $p\geq1$. Then $\SO_0(p-1, q)$ is not
  spherical in $\SU(p,q)$.
\end{lemma}

\begin{proof} Write $\C^{p+q} = \C \oplus \C^{p+q-1}$ and decompose
  vectors $v= v_1 + v_2$ accordingly.  Denote by $\la\cdot, \cdot\ra$
  the complex symmetric bilinear form on $\C^{p+q-1}$ which defines
  $\SO_0(p-1, q)$.  The following four real valued functions are
  $H$-invariant function are independent:
  \begin{eqnarray*} f_1(v) & :=& \re v_1\\
    f_2(v)&:= &\im v_1\\
    f_3(v)&:=& \re \la v_2, v_2\ra \\
    f_4(v)&:=&\im \la v_2, v_2\ra
  \end{eqnarray*}
  Hence each $H$-orbit on $\C^{p+q}$ has real codimension at least
  $4$ by Lemma \ref{independent}, and hence $H$ is not spherical by
  Lemma \ref{orbit bound SU}.
\end{proof}

\subsection{Type I maximal subgroups}

Let $H\subset G$ be a maximal subgroup of type I. Then
$H_\C=S(\GL(n_1,\C) \times \GL(n_2,\C))$, $n_i>0$, is a symmetric
subgroup of $G_\C$. In view of Lemma \ref{real symmetric} and Berger's
list \cite{Berger} we thus obtain:

\begin{lemma} \label{typeI-su}The maximal connected subgroups of Type
  I for $G=\SU(p,q)$ are given, up to conjugation, by the symmetric
  subgroups
  \begin{enumerate}
  \item $\mathrm{S}(\Unitary(p_1, q_1)\times \Unitary(p_2, q_2))$ with
    $p_1+p_2=p $ and $q_1 +q_2=q$.
  \item $\GL(1,\R)\SL(p,\C)$ if $q=p$.
  \end{enumerate}
\end{lemma}

\begin{lemma} \label{typeI-slmH} The maximal connected subgroups of
  Type I for $G=\SL(n,\Hb)$ are given, up to conjugation, by the
  symmetric subgroups:
  \begin{enumerate}
  \item $\mathrm{S}(\GL(n_1,\Hb) \times \GL(n_2,\Hb))$ with
    $n_1+n_2=n $.
  \item $\Unitary(1) \SL(n,\C)$.
  \end{enumerate}
\end{lemma}

\subsection{Type II maximal subgroups}

In this case we have $\C^n=\C^r \otimes \C^s$ with $2\leq r,s$ and 
$H_\C=\SL(r,\C) \otimes \SL(s,\C)$. 
In particular, $\dim H = r^2 + s^2 -2$. 

\subsubsection{The case of $G=\SU(p,q)$}

\begin{lemma}\label{typeII-su} Let $G=\SU(p,q)$ with $q\geq p\geq  1$. 
  Then type II real spherical subgroups $H\subset G$ occur only for
  $p=q=2$ and are given, up to conjugation, by:
  \begin{enumerate}
  \item $H=\SU(2) \otimes \SU(1,1)$,
  \item $H=\SU(1,1)\otimes \SU(1,1)$.
  \end{enumerate}
  Both cases are symmetric.
\end{lemma}

\begin{proof} In case $p=1$ all maximal reductive algebras which are
  real spherical are symmetric by \cite{KS1}. This prevents in
  particular type II real spherical subgroups. Henceforth we thus
  assume that $p\geq 2$.

  The local isomorphism $\SU(2,2)\simeq \SO_0(2,4)$ carries the
  subgroups (1) and (2) to $\SO_0(2,1)\times \SO(3)$ and
  $\SO_0(1,2)\times\SO_0(1,2)$, respectively. Hence they are symmetric
  and real spherical.

  Let $p+q=rs$ and $H_\C=\SL(r,\C) \otimes \SL(s,\C)$.  By
  Remark~\ref{rem:typeII-complexsubgroup} we can exclude that $H$ is
  complex, and hence we may assume that $H=H_1 \otimes H_2$ with
  $H_1, H_2$ real forms of $\SL(r,\C)$ and $\SL(s,\C)$.  We begin with
  the case where exactly one $H_i$, say $H_1$ is unitary:
  $H_1=\SU(p_1,q_1)$.  Let us first exclude the case where
  $H_2=\SL(s,\R)$. Note $s\geq 3$ as $\SL(2,\R) \simeq \SU(1,1)$ is
  unitary. Then the maximal compact subgroup $K_2:=\SO(s,\R)$ of $H_2$
  acts irreducibly on $\C^s$, and hence $V=\C^r \otimes\C^s$ is
  irreducible for the subgroup $H_1\otimes K_2$. Now $\C^s$ carries a
  positive definite $K_2$-invariant Hermitian form, and hence
  $H_1\otimes K_2$ leaves a Hermitian form of signature $(p_1s,q_1s)$
  invariant.  According to Lemma \ref{signature lemma} this form needs
  to be proportional to the original form coming from $G=\SU(p,q)$
  with signature $(p,q)$. It follows that the $K_2$-invariant form on
  $\C^s$ then has to be invariant under $H_2$ as well. This is
  impossible as $H_2$ is not compact. Likewise we can argue when
  $H_2=\SL(k,\Hb)$ which has maximal compact subgroup $K_2=\Sp(k)$
  acting irreducibly on $\C^{2k}$.  Similar to that we can argue with
  both $H_i$ either $\SL(\cdot,\R)$ or $\SL(\cdot, \Hb)$.

  Finally we need to turn to the case where $H_1=\SU(p_1,q_1)$ and
  $H_2=\SU(p_2,q_2)$, with $r=p_1+q_1$ and $s=p_2+q_2$.  We may assume
  that $p_1\leq q_1$ and $p_2\geq q_2$.  Then
  $(p_1-q_1)(p_2-q_2)\leq 0$ and hence
  \[p_1p_2+q_1q_2\leq p_1q_2+p_2q_1.\] We now exploit that
  $H_1\otimes H_2$ leaves invariant on $\C^{p+q}=\C^r\otimes\C^s$ both
  the defining form of $\SU(p,q)$ and the tensor product of the
  defining forms of $\SU(p_1,q_1)$ and $\SU(p_2,q_2)$. By comparing
  signatures (cf. Lemma \ref{signature lemma}) we thus obtain
  \[(p,q)=(p_1p_2 + q_1 q_2, p_1 q_2 + p_2 q_1).\] With that the
  dimension bound \eqref{dimsu} reads
  \[ r^2 + s^2-2 \geq 2(p_1p_2+q_1q_2)(q_1p_2 + q_2p_1) - (p_1 p_2
    +q_1q_2)\] or
  \begin{equation}\label{dimbd}
    r^2+s^2-2\geq2p_2q_2(p_1^2+q_1^2)+2p_1q_1(p_2^2+q_2^2)-(p_1p_2+q_1q_2).
  \end{equation} 
  Now we distinguish various cases.

  We first assume $p_1,q_1,p_2,q_2$ are all non-zero.  If they are all
  $1$ then we are in case (2), hence we may assume $q_1\ge 2$ or
  $p_2\ge 2$. By symmetry between $r$ and $s$ we can assume the
  latter. With $(x+y)^2\le 2(x^2+y^2)$ our bound \eqref{dimbd} implies
  \[ r^2+s^2-2\geq p_2 q_2 r^2 + p_1 q_1 s^2 - (p_1p_2 +q_1q_2)\] and
  hence, since $p_1q_1\ge 1$ and $p_2q_2\ge 2$,
  \[ r^2+s^2-2\geq r^2+s^2+\frac12 p_2q_2 r^2- (p_1p_2 +q_1q_2).\] As
  $r\ge 2$ we find
  \[\frac12 p_2q_2 r^2\ge p_2q_2r=p_2q_2(p_1+q_1)\ge p_2p_1+q_2q_1\]
  and reach a contradiction.

  Hence we may assume now that $p_1=0$ or $q_2=0$. By symmetry between
  $r$ and $s$ we can assume the former, that is
  \[H=\SU(r) \otimes \SU(p_2,q_2)\] with $p_2+q_2=s$. The bound
  \eqref{dimbd} now reads:
  \begin{equation}\label{tensor1} r^2 + s^2-2\geq 2p_2 q_2 r ^2 - r
    q_2.
  \end{equation}

  If $s=2$ then $p_2=q_2=1$ and \eqref{tensor1} gives
  $r^2+2\geq 2r^2-r$, from which it follows that $r=2$ and we are in
  case (1).

  Hence we can assume $s>2$ and $p_2q_2\ge 2$.  As $q_2\leq s$ we
  obtain from \eqref{tensor1} that $r^2+s^2> 4r^2-rs$.  It easily
  follows that $s>r$.

  Now we use that $H$ has an orbit of real codimension at most $3$ on
  $\C^r \otimes \C^s$ (see Lemma \ref{orbit bound SU}). This implies
  that $H$ has an orbit of real codimension at most $3$ on
  $\C^r \otimes (\C^s)^*=\Mat_{r,s}(\C)$ with the action of $H$ given
  as follows: $(h_1,h_2)\cdot X = h_1 X h_2^{-1}$.  Let $\Herm(r, \C)$
  denote the space of Hermitian matrices of size $r$, and for $k,l>0$
  let
  \[
    I_{k,l}=\left(\begin{matrix} I_k&0\\0&-I_l
      \end{matrix}\right).
  \]
  The map
  \begin{equation} \label{matrix submersion}\Phi: \Mat_{r,s}(\C) \to
    \Herm(r, \C), \ \ X\mapsto X I_{p_2,q_2} X^*
  \end{equation}
  is submersive and satisfies
  $\Phi(h_1Xh_2^{-1})= h_1\Phi(X) h_1^{-1}$ for $h_1\in\SU(r)$ and
  $h_2\in\SU(p_2,q_2)$.  Hence there must be an $\SU(r)$-orbit on
  $\Herm(r, \C)$ of real codimension at most $3$ and therefore
  $r\leq 3$ by the spectral theorem.

  We are now left with the examination of the cases where
  $H=\SU(r) \otimes \SU(p_2, q_2)$ with $r=2, 3$ and $s>r$. Set
  $H':=\1 \otimes \SU(p_2,q_2)\subset H$. Then since $H$ has an open
  orbit on $G/P$, $H'$ must have an orbit of codimension at most
  $r^2-1$.

  To move on we introduce projective type coordinates for the flag
  variety $G/P$. We can describe a flag $\F\in G/P$ as follows
  \[\F:\, \la v_1\ra \subset \la v_1, v_2\ra \subset
    \ldots \subset \la v_1, \ldots, v_p\ra\] such that
  $\{ v_1, \ldots, v_p\}$ is orthonormal with respect to the standard
  Hermitian scalar product on $V=\C^n$. Observe in addition that all
  $v_i$ are isotropic and mutually orthogonal with respect to the form
  $(\cdot, \cdot)_{p,q}$. It is important to note that $\F$ determines
  the $v_i$ uniquely up to scaling with $\Unitary(1)$.

  Decompose $V= \C^s \oplus\ldots\oplus \C^s$ into $H'$-orthogonal
  summands where we have two or three summands according to $r=2$ or
  $r=3$.  This gives us $r$ projections $\pi_j : \C^n \to \C^s$.
  Likewise for every $1\leq m\leq p$ the $\pi_j$ induce projections
  $\bigwedge^m \C^n \to \bigwedge^m \C^s$ which will be also denoted
  by $\pi_j$.  Further, the invariant form $(\cdot, \cdot)_{p,q}$
  induces an invariant form on $\bigwedge^m \C^n$, denoted by the same
  symbol.
 
  We define functions $g_{mjk}$ on $G/P$ for $1\leq m\leq p$ and
  $1\leq j,k\leq r$ by
  \[ g_{mjk} (\F):= (\pi_j(v_1\wedge\ldots\wedge v_{m}),
    \pi_k(v_1\wedge\ldots\wedge v_{m}))_{p,q} \, .\] Note, that for
  fixed $m$, the rational functions
  \[ f_{mjk}:= \re \left(\frac{g_{mjk}}{g_{m11}} \right)\quad
    \hbox{and}\quad f_{mjk}':= \im
    \left(\frac{g_{mjk}}{g_{m11}}\right)\] are all $H'$-invariant.
  Already for $m=1$, we obtain $r^2-1$ independent functions this way.
  Further, as $p\ge 2$ and $n>4$ we obtain at least one independent
  invariant for $m=2$ (it will depend on the non-trivial
  $v_2$-coordinate of $\F$), which gives a contradiction by Lemma
  \ref{independent}.
\end{proof}

\subsubsection{The case of $G=\SL(m,\Hb)$.}

\begin{lemma} \label{typeII-slmH} Let $G=\SL(m,\Hb)$ with $m\geq 2$.
  The only type II real spherical subgroup $H\subset G$ occurs for
  $m=2$ and is given, up to conjugation, by:
  \[H=\SU(1,1)\otimes \SU(2)\] This is a symmetric subgroup.
\end{lemma}

\begin{proof} For $H_\C=\SL(r,\C)\otimes\SL(s,\C)$ the dimension bound
  \eqref{dimsl} reads
  \[ r^2 +s^2 - 2 \geq 2m^2 -2m.\] The equation $rs=2m$ together with
  $r,s\ge 2$ gives $r+s\le m+2$ and hence $r^2+s^2\le m^2+4$. Hence
  $2m^2-2m\le m^2+2$, which implies $m=2$.  Then $r=s=2$. The local
  isomorphism $\SL(2,\Hb)\simeq \SO_0(1,5)$ carries
  $H=\SU(1,1)\otimes \SU(2)$ to $\SO_0(1,2)\times\SO(3)$, which is
  symmetric.  On the other hand, $H=\SU(1,1)\otimes\SU(1,1)$ is
  excluded by the rank inequality.
\end{proof}

\subsection{Type III maximal reductive subgroups}

Here $H_\C$ is simple and acts irreducibly on $\C^n$.  In the
following we denote by $\Sym(m,\C)$ and $\Skew(m,\C)$ the space of
symmetric, respectively skew-symmetric, matrices of size $m$.

\subsubsection{The case of $G=\SU(p,q)$}

\begin{lemma} \label{typeIII-su} Let $p + q \geq 3$ and let
  $H\subset\SU(p,q)$ be a reductive real spherical subgroup of type
  III. Then, up to conjugation, $H$ is one of the following symmetric
  subgroups:
  \begin{enumerate}
  \item $\SO_0(p,q)$.
  \item $\SO^*(2p)$ if $p=q$.
  \item $\Sp(p/2, q/2)$ if $p,q$ are even.
  \item $\Sp(p, \R)$ if $p=q$.
  \end{enumerate}
\end{lemma}

\begin{proof} According to \cite{KS1}, the assertion is true for $p=1$
  and henceforth we assume that $q\ge p\ge 2$. By the dimension
  bound~\eqref{dimsu} we have for $n=p+q$
  \begin{equation}\label{bound 4n-10}
    \dim H\ge 2pq - p = 2p(n-p)-p\ge 4n-10,
  \end{equation}
  where the last inequality follows since $2\le p\le \frac n2$.  We
  recall from Lemma \ref{orbit bound SU} that $V=\C^n$ is a
  prehomogeneous vector space for $H_\C$, and since $V$ is irreducible
  and $H_\C$ is simple, we can apply
  Proposition~\ref{prop:prehomogeneous} and
  Table~\ref{tab:prehomogeneous} as explained in Remark \ref{remark on
    forms}.

  \bigskip \noindent {\it $\bullet$ $H_\C= \SL(m,\C)$ acting on
    $V=\Skew(m,\C)$}, $m\ge 5$.  Here $n=\frac12 m(m-1)$ and
  $\dim H=m^2-1$. Hence by \eqref{bound 4n-10} we obtain
  $m^2-2m-9\le 0$ which is excluded for $m\ge 5$.

  \bigskip\noindent{ \it $\bullet$ $H_\C= \SL(m,\C)$ acting on
    $V=\Sym(m,\C)$}, $m\ge 3$.  Here $n=\frac12m(m+1)$ and we get
  $m^2+2m-9\le 0$, which is excluded with $m\ge 3$.

  \bigskip\noindent{ \it $\bullet$ $H_\C=\SO(n,\C)$ acting on
    $V=\C^n$.}  This leads to (1) and (2).

  \bigskip\noindent{ \it $\bullet$ $H_\C=\Sp(m,\C)$ acting on
    $\C^n=\C^{2m}$.}  This leads to (3) and (4).

  \bigskip\noindent{ \it $\bullet$ The sporadic prehomogeneous vector
    spaces.}  Since we assume that $p\geq 2$, the dimension bound
  gives no possibilities.
\end{proof}

\subsubsection{The case of $G=\SL(m,\Hb)$}

\begin{lemma} \label{typeIII-slmH} Let $H\subset\SL(m,\Hb)$ for
  $m \geq 3$ be a real spherical subgroup of type III. Then $H$ is
  conjugate to one of the following symmetric subgroups:
  \begin{enumerate}
  \item $\SO^*(2m)$
  \item $\Sp(p,q)$, $p + q = m$.
  \end{enumerate}
\end{lemma}
 
\begin{proof}

  Let $V = \C^{2m} = \C^n$. By~\eqref{dimsl} a spherical subgroup $H$
  satisfies \[ \dim H \geq 2m^2 - 2m > 2m = n = \dim_\C V, \] as
  $m > 2$. Since $H_\C$ acts via $\rho$ irreducibly on $V$, it follows
  from Lemma \ref{orbit bound SU*} that the triplet $(H_\C,\rho,V)$
  appears among the even-dimensional cases in
  Proposition~\ref{prop:prehomogeneous}.  In particular we do not have
  to consider the odd dimensional cases (10) and (22) from Table
  \ref{tab:prehomogeneous}. Further, via Remark \ref{remark on forms},
  we can eliminate the cases (1), (6), (8), (9), (12) - (14), (16) -
  (18), (20), (21) and (23) from Table \ref{tab:prehomogeneous}.
  Since $H$ has to be proper, case (2) is excluded as well.  This
  leaves us with the following possibilities:

  \bigskip\noindent{ \it $\bullet$ $H_\C = \SL(k,\C)$, acting on
    $V = \Skew(k,\C)$} with $k\ge 5$. The dimension bound for $H$
  reads
  \begin{equation}\label{dimbound_Ak_in_SLmH}
    k^2 - 1 \geq 2m^2 - 2m.
  \end{equation} 

  Since $2m = \frac12{k(k-1)}$ and $k \geq 5$, we have
  $m \geq k \geq 5$. Furthermore, $k^2=4m + k$ and by
  \eqref{dimbound_Ak_in_SLmH} we get the contradiction
  \[ 4m+k - 1 \geq 2m(m - 1) \geq 8m. \]

  \bigskip\noindent{ \it $\bullet$ $H_\C =\SL(k,\C)$, acting on
    $V = \Sym(k,\C)$} with $k\ge 3$.  Here $2m =
  \frac12{k(k+1)}$. Since $k \geq 3$, we have $m \geq k \geq
  3$. Furthermore, $k^2=4m - k$ and by \eqref{dimbound_Ak_in_SLmH} we
  get the contradiction
  \[ 4m-k - 1 \geq 2m(m - 1) \geq 4m. \]

  \bigskip\noindent{ \it $\bullet$ $H_\C =\SO(2m,\C)$ acting on
    $V = \C^{2m}$}. The real form $H = \SO^*(2m)$ gives case (1) of
  the lemma. The real form $H = \SO_0(p,q)$, $p + q = 2m$ cannot
  occur, since its maximal compact subgroup $\SO(p)\times\SO(q)$ must
  be conjugate to a subgroup of $K=\Sp(m)\subset \SU(m,m)$ from which
  we conclude $p = q = m$.  But then,
  $\mathrm{rank}_\R(H) = m > m - 1 = \mathrm{rank}_\R(G)$.

  \bigskip\noindent{\it $\bullet$ $H_\C =\Sp(m,\C)$ acting on
    $V = \C^{2m}$}. The real form $H = \Sp(p,q)$ with $p + q = m$
  gives case (2) of the lemma. The real form $H = \Sp(m,\R)$ does not
  occur, since its real rank equals $m$ which is greater than
  $\mathrm{rank}_\R(G) = m -1$.

  \bigskip\noindent{ \it $\bullet$ $H_\C =\SL(k,\C)$ acting on
    $V = \bigwedge^3 \C^k$, $k = 7,8$}.  It is easy to see that for
  $k = 8$ the dimension bound is violated, while for $k = 7$ the
  dimension of $V$ is odd.

  \bigskip\noindent{ \it $\bullet$ $H_\C = \Spin(k,\C)$ acting on a
    half spin representation, $k = 10, 14$}.  The representation
  spaces are $\C^{16}$ and $\C^{64}$ respectively.  The dimension
  bound for $H$ reads $\frac12 k(k-1)\ge 2m(m-1)$, whence we get the
  contradiction $k\ge 2m$.
\end{proof}

This concludes the proof of Theorem \ref{maximals are real forms} for
$G_\C=\SL(n,\C)$.

\section{Maximal reductive real spherical subgroups for the
orthogonal  groups}\label{max SO}

We prove the statement in Theorem \ref{maximals are real forms} for
$\gf_\C=\so(n,\C)$, assuming $n\ge 5$ throughout.  We may assume again
that $\hf_\C$ is maximal reductive (cf.~Corollary \ref{Komrakov cor}).

\subsection{The real forms}
Let $G_\C=\SO(n,\C)$. Our focus is on the real forms $G=\SO_0(p,q)$
with $p+q=n$ and $p\leq q$, and $G=\SO^*(2m)$ with $n=2m$.

Note that $\so^*(6)\simeq \su(1,3)$ was already treated in Lemmas
\ref{typeI-su}, \ref{typeII-su}, and \ref{typeIII-su}. Furthermore,
$\so^*(8)\simeq \so(2,6)$ will be treated below through the general
case of $\SO_0(p,q)$.  We may thus assume $m\ge 5$ for $\SO^*(2m)$.

The dimension bounds obtained from \eqref{db} and the cited table of
\cite{Helgason} read:
\begin{equation}\label{dimso} \dim H \geq pq - p \qquad
  (G=\SO_0(p,q)),\end{equation}
\begin{equation}\label{dimsoev} \dim H \geq m^2 -\frac32 m\qquad
  (G=\SO^*(2m)).\end{equation}

For further reference we record the matrix realizations of $G$ and
$P$.  We begin with $G=\SO_0(p,q)$ which we consider as the invariance
group of the symmetric form $\la\cdot, \cdot\ra_{p,q}$ defined by
\[\begin{pmatrix} 0 & I_p & 0\\ I_p & 0 & 0 \\ 0 & 0 &
    I_{q-p}\end{pmatrix}\, .\] Accordingly we obtain for the Lie
algebra
\[\so(p,q)=\left\{
    \begin{pmatrix} A & B & E \\ C & -A^T& F \\ -F^T & - E^T &
      D\end{pmatrix} \Bigg| \quad
    \begin{matrix}
      A, B, C\in \Mat_{p, p} (\R),\\
      E, F \in \Mat_{p, q-p}( \R), \\
      D \in \Mat_{q-p, q-p}(\C),\\
      B^T, C^T, D^T=-B, -C, -D\end{matrix}\,\right\}\, .\]

We choose the minimal parabolic such that
\[\pf=\left\{\begin{pmatrix} A & B & E \\ 0 & -A^T& 0 \\ 0 & - E^T &
      D\end{pmatrix}\in \so(p,q)\mid A \ \hbox{upper
      triangular}\right\}
\]
so that $P$ stabilizes the isotropic real flag
$\la e_1\ra \subset \la e_1, e_2\ra \subset\ldots\subset \la e_1,
\ldots, e_p\ra $ in $\R^n$.  Moreover
\[G/P\simeq\{ V_1\subset \ldots \subset V_p\mid V_p \subset \R^n, \
  \dim_\R V_i = i, \la V_p, V_p\ra _{p,q}=\{0\}\}\] is the variety of
full isotropic $p$-flags in $\R^{p+q}$ with respect to the symmetric
bilinear form $\la \cdot, \cdot\ra _{p,q}$.

Let us denote
\[\Nc_{p,q}^\R:= \{ [v]\in \Pb(\R^n)\mid \la v,v\ra_{p,q}=0\}\]
and note that there is a $G$-equivariant surjective map
$G/P \to \Nc_{p,q}^\R$.  Hence we obtain the following lemma.

\begin{lemma} \label{so-codim2}Let $G=\SO_0(p,q)$ and $H\subset G$ a
  real spherical subgroup. Then there exists an $H$-orbit on $\R^n$ of
  codimension at most 2.
\end{lemma}

\begin{proof} The fact that $H$ has an open orbit on $G/P$ implies
  that there is an open $H$-orbit in $\Nc_{p,q}^\R$.
\end{proof}

We continue to recall a few structural facts for the group
$\SO^*(2m)$.  We identify $\Hb^m$ with $\C^{2m}$ via
$\Hb^m= \C^m \oplus j \C^m$.  Denote by $h\mapsto \oline h$ the
conjugation on $\Hb^m$. The group $\SO^*(2m)$ consists of the right
$\Hb$-linear transformations on $\Hb^m$ which preserve the
$\Hb$-valued form
\[ \phi (h, h')= \oline{h_1} j h_1' + \ldots + \oline{h_m} j h_m'
  \qquad (h_i, h_i'\in \Hb)\, .\]

Observe that $\phi$ is a so-called skew-Hermitian form, i.e.~it is
sesquilinear and skew.  We recall that {\it sesquilinear} means
$\phi(hx,h'x')= \oline x \phi (h,h') x'$ for all $h,h'\in\Hb^m$,
$x,x'\in\Hb$, and {\it skew} refers to
$\oline{\phi(h,h')}= -\phi(h',h)$.

Denote the $\C$-part of $\phi(h,h')$ by $(h,h')_{m,m}\in\C$ and the
$j\C$-part by $\la h,h'\ra\in\C$.  If we write elements $h\in \Hb^m $
as $h = x + jy$ with $x,y\in\C^m$, then
\[ ( h, h')_{m,m}= \oline y^T x' - \oline x^T y'\] and
\[ \la h, h'\ra = x^T x' + y^T y' \, .\] Notice that
$i(\cdot, \cdot)_{m,m}$ is a Hermitian form of signature $(m,m)$.  In
particular, if we view $\SO^*(2m)$ as a subgroup of $\SL(2m,\C)$, then
$\SO^*(2m)= \SO(2m,\C) \cap \SU(m,m)$.

The minimal flag variety is given by isotropic right $\Hb$-flags
\begin{equation}\label{H-flags I} G/P =\{ V_1 \subset \ldots \subset
  V_{[m/2]}
  \subset\Hb^{m} \mid  \phi(V_i, V_i) =\{0\},\dim_\Hb V_{i}=i
  \} \, .\end{equation}

\begin{remark} Observe that the sesquilinear form $\phi$ is uniquely
  determined by its $\C$-part or $j\C$-part. Hence an $\Hb$-subspace
  $V_i\subset \Hb^m$ is isotropic if and only if it is isotropic for
  $\la \cdot, \cdot\ra$ (or $(\cdot, \cdot)_{m,m}$).  Recall that
  $G_\C/B_\C$ is the variety of $\la\cdot,\cdot\ra$-isotropic (left)
  $\C$-flags in $\C^{2m}=\Hb^m$.  Hence the right hand side of
  \eqref{H-flags I} embeds totally real into the quotient of
  $G_\C/B_\C$ consisting of even-dimensional isotropic complex flags.
  A simple dimension count then shows equality in \eqref{H-flags I}.
\end{remark}

\subsection{Type I maximal subgroups}

Let $H\subset G$ be a maximal subgroup of type I. Then
$H_\C=\SO(n_1,\C) \times \SO(n_2,\C)$, $n_i>0$ and $n=n_1+n_2$, or
$H_\C=\GL(n/2, \C)$ for $n$ even. In both cases $H_\C$ is a symmetric
subgroup of $G_\C$. Hence with Lemma \ref{real symmetric} and Berger's
list \cite{Berger} we obtain:

\begin{lemma}\label{typeI-so*} Let $H\subset\SO^*(2m)$ be a  subgroup of type I. 
  Then $H$ is symmetric, and up to conjugation it equals one of the
  following groups:
  \begin{enumerate}
  \item $\SO^*(2m_1) \times \SO^*(2m_2)$ with $m_1+m_2=m$,
    $m_1, m_2>0$,
  \item $\SO(m,\C)$,
  \item $\GL(m/2,\Hb)$ for $m$ even,
  \item $\operatorname{U}(k, l)$ with $k+l=m$.
  \end{enumerate}
\end{lemma}

\begin{lemma} \label{typeI-so}Let $H\subset\SO_0(p,q)$ be a subgroup
  of type I.  Then $H$ is symmetric, and up to conjugation it equals
  one of the following groups:
  \begin{enumerate}
  \item $\SO_0(p_1,q_1) \times \SO_0(p_2,q_2)$ with $p_1+p_2=p$ and
    $q_1+q_2=q$,
  \item $\SO(p,\C)$ for $p=q$,
  \item $\GL(p,\R)$ with $p=q$.
  \item $\operatorname{U}(p/2, q/2)$ for $p, q$ even.
  \end{enumerate}
\end{lemma}

\subsection{Type II maximal reductive subgroups}

Here we suppose that $H_\C$ is a maximal reductive subgroup of
$G_\C=\SO(n,\C)$, $n\geq 5$ of type II. Hence there are the two
possibilities:
\begin{itemize}
\item $H_\C = \SO(r,\C) \otimes \SO(s,\C)$ with $rs=n$,
  $3\leq r\le s$, and $r, s\neq 4$.
\item $H_\C= \Sp(r, \C)\otimes \Sp(s,\C)$ with $4rs=n$ and
  $1\leq r\leq s$.
\end{itemize}

\subsubsection{The case of $G=\SO^*(2m)$}

\begin{lemma} \label{typeII-so*}Let $G=\SO^*(2m)$ for $m\geq 5$.  Then
  there exist no real spherical subgroups $H\subset G$ of type II.
\end{lemma}

\begin{proof} When $m\geq 5$ no type II subgroup satisfies the
  dimension bound \eqref{dimsoev}.
\end{proof}

\subsubsection{The case of $G=\SO_0(p,q)$}

\begin{lemma}\label{typeII-so} Let $G=\SO_0(p,q)$ with $p+q\geq 5$. 
  Then type II real spherical subgroups $H\subset G$ occur only for
  $p=q=4$ and are given, up to conjugation, by:
  \begin{enumerate}
  \item $H=\Sp(1,\R) \otimes \Sp(2,\R)$,
  \item $H=\Sp(1) \otimes \Sp(1,1)$.
  \end{enumerate}
  Both cases are symmetric.\end{lemma}

\begin{proof} We first prove that the groups listed under (1) and (2)
  are symmetric and hence real spherical.  Write $H=H_1\otimes H_2$
  and $\C^8=\C^2 \otimes \C^4$.  The symplectic forms $\Omega_i$ on
  $\C^{2i}$ defined by $H_i$ give rise to the $\SO_0(4,4)$-invariant
  symmetric form $\la\cdot, \cdot\ra=\Omega_1 \otimes \Omega_2$ on
  $\C^8$.  Write $J_1$ and $J_2$ for the matrices defining $\Omega_1$
  and $\Omega_2$.  Then
  $B=\Omega_1(J_1 \cdot , \cdot) \otimes \Omega_2(J_2\cdot, \cdot)$
  defines a symmetric bilinear form on $\C^8$ and we write
  $g\mapsto g^t$ for the corresponding transpose on matrices. Then the
  assignment $g\mapsto ( J_1\otimes J_2) g^{-t} (J_1\otimes J_2 ) $
  defines an involution on $G=\SO_0(4,4)$ with fixed group $H$. Hence
  $H$ is symmetric (and outer isomorphic to
  $\SO_0(2,1)\times\SO_0(2,3)$, respectively $\SO(3)\times\SO_0(1,4)$,
  from Berger's list).

  Let $H=H_1 \otimes H_2$ be a type II subgroup. We consider first the
  case where each $H_{i\C}$ is symplectic.  We start with
  $H=\Sp(r,\R) \otimes \Sp(s, \R)$.  The invariant Hermitian form on
  each factor gives an invariant Hermitian form on the tensor product
  with signature $(2rs, 2rs)$ which then must be equal to $(p,q)$.
  The dimension bound \eqref{dimso} becomes
  \[r(2r+1)+s(2s+1)\ge 4r^2s^2-2rs.\] For $r\ge 1$ and $s\ge 2$ we
  have $r(2r+1)\le 3r^2 \le \frac34 r^2s^2$ and
  $s(2s+1)\le \frac52 r^2s^2$. It follows that
  $4r^2s^2-2rs\le \frac{13}4r^2s^2$ which easily implies $rs<3$. Since
  $4rs=n\ge 5$ it follows that $r=1$ and $s=2$. These data produce the
  first symmetric subgroup mentioned in the lemma.

  For $H=\Sp(r,\R) \otimes \Sp(p_2,q_2)$ we obtain the same signature
  condition $p=q=2rs$ as before and hence
  $H=\Sp(1,\R)\otimes \Sp(1,1)$.  Up to an outer automorphism this is
  a real form of a symmetric subgroup in $G_\C$, which can be excluded
  with Berger's list for $G=\SO_0(4,4)$.

  The case where $H=\Sp(p_1, q_1)\otimes\Sp(p_2,q_2)$ with $r=p_1+q_1$
  and $s=p_2+q_2$ is treated analogously as Lemma \ref{typeII-su}.  We
  can assume $p_1\le q_1$ and $p_2\ge q_2$.  The group $H$ leaves
  invariant a Hermitian form of signature
  $(4(p_1 p_2 + q_1q_2), 4 (p_1q_2 + p_2 q_1))$, which must then equal
  $(p,q)$ by Lemma \ref{signature lemma}. Then the dimension bound
  \[r(2r+1)+s(2s+1)\ge 16p_2q_2(p_1^2+q_1^2) +16p_1q_1(p_2^2+q_2^2)
    -4(p_1p_2+q_1q_2)\] leads to the absurd unless $p_1=0$ and
  $H=\Sp(r)\otimes\Sp(p_2,q_2)$ with $r\le s$.  Using a matrix
  submersion as \eqref{matrix submersion} we obtain with Lemma
  \ref{so-codim2} that $r=1$, hence $H=\Sp(1)\otimes \Sp(p_2,q_2)$ and
  $G=\SO_0(4p_2,4q_2)$.  Set $H':= \Sp(p_2,q_2)\subset H$.  Then $H'$
  is of codimension 3 in $H$ and thus $H'$ admits an orbit of
  codimension 3 on $G/P$.  We parameterize flags $\F\in G/P$ as in
  Lemma \ref{typeII-su}. Let $V=\R^{4p_2+4q_2}\simeq \C^{2p_2+2q_2}$.
  First we note that there are three independent real symplectic forms
  which are invariant under $H'$. In fact, if $\Omega$ is the complex
  symplectic form on $\C^{2p_2+2q_2}$ which defines $H'$, then
  $\Omega_1 =\re \Omega$, and $\Omega_2=\im \Omega_2$ give two
  independent symplectic forms. A third form is given by
  $\Omega_3=\im (\cdot, \cdot)_{2p_2, 2q_2}$.  Concretely, the
  $\Omega_i$ are given as follows: Out of the standard symplectic
  forms $J_i$ on $\R^4$
  \[
    J_1=\begin{pmatrix} 0 & 0 & 1 & 0 \\ 0 & 0 & 0 & 1 \\ -1 & 0 & 0 &
      0 \\ 0 & -1 & 0& 0 \end{pmatrix}\quad
    J_2=\begin{pmatrix} 0 & 0 & 0& 1 \\ 0 & 0 & 1 & 0 \\ 0& -1 & 0 & 0 \\
      -1 & 0 & 0& 0 \end{pmatrix}\quad J_3=\begin{pmatrix} 0 & 1 & 0 &
      0 \\ -1 & 0 & 0 & 0 \\ 0& 0 & 0 & 1 \\ 0 & 0 & -1&
      0 \end{pmatrix}\] we build the forms
  ${\mathbf J_i} = \diag (J_i) \in \Mat_{p+q}(\R)$.  Then
  $\Omega_i ( \cdot, \cdot) = \la {\mathbf J_i}\cdot, \cdot\ra_{p,q}$.

  This gives us two independent rational invariants
  \[f_i(\F):= \frac{\Omega_i(v_1 \wedge v_2)}{\Omega_3(v_1\wedge
      v_2)}\qquad (i=1,2)\, .\] Further invariants we obtain via
  \[ g_{j,k}(\F)= \frac{(\Omega_j\wedge \Omega_k) (v_1 \wedge
      v_2\wedge v_3\wedge v_4)} {(\Omega_3\wedge \Omega_3)(v_1\wedge
      v_2\wedge v_3\wedge v_4)}\qquad (1\leq j\leq k\leq 3, (j,k)\neq
    (3,3))\, .\] Clearly each $g_{j,k}$ is independent to
  $\{f_1,f_2\}$ as the $f_i$ only depend on the first two coordinates
  $v_1,v_2$. Moreover, if $p_2>1$ we obtain additional invariants with
  an analogous construction on $\bigwedge^6 V$. Thus for $p_2>1$ we
  obtain at least $4$ algebraically independent $H'$-invariants on
  $G/P$ contradicting the fact that the generic $H'$-orbit is of
  codimension at most $3$ (cf.~Lemma \ref{independent}).  This leaves
  us to investigate the case with $p_2=1$.  Now if $q_2=1$, the
  $g_{j,k}$ are all dependent and $H$ is the second symmetric subgroup
  mentioned in the lemma.  If $q_2>1$, then we obtain at least 4
  algebraically independent functions out of $f_i, g_{jk}$. To verify
  that we may restrict ourselves to the case $q_2=2$.  We fix the
  first two coordinates of $\F$ to be $v_1=e_1, v_2=e_6$.  For
  $\lambda, \mu, \nu,\e,\delta \in\R$ we consider
  \[ \tilde v_3 = e_3 + \lambda e_4 + \mu e_7 + \nu e_8 + \e e_9
    \qquad \tilde v_4= e_3 - \mu e_7 + \delta e_{10}\, .\] Then
  $\{ v_1, v_2, \tilde v_3, \tilde v_4\} $ is a set of mutually
  orthogonal vectors with respect to $\la\cdot,\cdot
  \ra_{p,q}$. Moreover $\tilde v_3$, resp. $\tilde v_4$, is isotropic
  provided that $2(\mu + \lambda\nu) +\e^2=0$, resp.
  $ -2\mu + \delta^2=0$.  We choose now the parameters such that both
  $\tilde v_3$ and $\tilde v_4$ are isotropic.  Let $v_3$, $v_4$ be
  unit vectors obtained from $\tilde v_3, \tilde v_4$.
 
  This then gives us isotropic flags
  \[\F= \{ \la e_1\ra\subset \la e_1, e_6\ra\subset \la e_1, e_6,
    v_3\ra\subset \la e_1, e_6, v_3, v_4\ra\}\, .\] Then
  \[g_{11}(\F) = \frac{\Omega_1 (v_1, v_4)
      \Omega_1(v_3,v_2)}{\Omega_3(v_1,v_2)
      \Omega_3(v_3,v_4)}\quad\hbox{and} \quad g_{22}(\F) =
    \frac{\Omega_2(v_1, v_3) \Omega_2(v_2,v_4)} {\Omega_3(v_1,v_2)
      \Omega_3(v_3,v_4)}
  \]
  and in particular
  \[g_{11}(\F)= \frac{\lambda\mu}{\mu \lambda - \nu + \e \delta}
    \quad\hbox{and}\quad g_{22}(\F)=\frac{-\nu }{\mu \lambda - \nu +
      \e \delta} \,.\] It follows that
  $\{ f_1, f_2, g_{11}, g_{22}\} $ are independent and hence $H$ is
  not real spherical for $q_2>1$ (cf.~Lemma \ref{independent}).

  Next we look at the case where $H=H_1 \otimes H_2$ with both
  complexifications orthogonal.  We begin with both $H_i=\SO^*(2m_i)$
  quaternionic, and $m_i> 2$.  The invariant Hermitian form on each
  factor gives an invariant Hermitian form on the tensor product with
  signature $(2m_1m_2, 2m_1m_2)$ which then must be equal to $(p,q)$
  by Lemma \ref{signature lemma}.  Then the dimension bound
  \[2m_1^2-m_1+2m_2^2-m_2\ge 4m_1^2m_2^2- 2m_1m_2\] is easily seen to
  be violated.  Similarly if $H_1=\SO^*(2m_1)$ and
  $H_2=\SO_0(p_2,q_2)$ with $p_2+q_2=s$, then $p=q=m_1(p_2 +q_2)$ by a
  signature argument, and exactly the same bound as above results.

  This reduces to the final case where
  $H=\SO_0(p_1,q_1) \otimes \SO_0(p_2,q_2)$, which is treated
  similarly as the previous case of
  $H=\Sp(p_1, q_1)\otimes\Sp(p_2,q_2)$.  Comparing signatures we find
  that $p=p_1p_2+q_1q_2$ and $q=p_1q_2+p_2q_1$, and the dimension
  bound then implies $H=\SO(r)\otimes\SO_0(p_2,q_2)$ with $r\le s$. By
  applying a matrix submersion as \eqref{matrix submersion} we obtain
  with Lemma \ref{so-codim2} that $H$ must have an orbit on
  $\Sym(r,\R)$ of codimension $2$. This contradicts that $r\ge 3$.
\end{proof}

\subsection{Type III maximal subgroups}

We assume that $H_\C$ is simple and acts irreducibly on~$V$.

\subsubsection{The case $G=\SO_0(p,q)$}

\begin{lemma}\label{typeIII-so} Let $G=\SO_0(p,q)$ for $1\leq p\leq q$
  and $p+q\ge 5$. Then the only real spherical subgroups $H\subset G$
  of type III are given, up to isomorphism, by
  \begin{enumerate}
  \item $\hf=\sG_2^1$ in $\gf=\so(3,4)$.
  \item\label{so(3,4)} $\hf=\spin(3,4)$ in $\gf=\so(4,4)$ (two
    conjugacy classes swapped by an outer automorphism of $\gf$).
  \end{enumerate}
  These pairs are absolutely spherical and the second one is
  symmetric.
\end{lemma}

Note that although the pair of Lie algebras $(\gf,\hf)$ in
\ref{so(3,4)} is symmetric, this is not the case for the space
$G/H$, since the corresponding involution does not lift to $G$.
Nevertheless $G/H$ is real spherical since the existence of an open
$P$-orbit is a property of the Lie algebras.

\begin{proof} For $p=1$ it follows from \cite{KS1} that there are no
  such spherical subgroups.  The case $\so(2,3)$ is quasi-split and
  features no type III subalgebras according to Krämer.  Hence we may
  assume here $2\leq p\leq q$, $q\geq 3$ and $p+q>5$. Then
  \[(pq-p)-(p+q)=(p-2)(q-3)+p+q-6\ge 0.\] Hence the dimension bound
  \eqref{dimso} implies
  \begin{equation}\label{dimsoIII}
    \dim H \geq pq-p\ge p+q=\dim V.
  \end{equation}
  In particular, we can apply Proposition~\ref{prop:prehomogeneous}
  and Remark \ref{remark on forms}.  We observe also that $pq-p> p+q$
  if $p+q>6$.

  \bigskip\noindent{ \it $\bullet$ $H_\C$, adjoint representation.}
  Then $\dim H=\dim V$, which is excluded unless $\dim H=6$, by the
  strictness of \eqref{dimsoIII}. Then $H=\SO_0(3,1)$ and $H_\C$ is
  not simple.

  \bigskip \noindent{ \it $\bullet$ $H_\C =\SO(m,\C)$ acting on
    $V = \C^{m}$}. This is possible, but then we would have
  $H_\C = G_\C$.

  \bigskip\noindent{ \it $\bullet$ $H_\C = \Sp(m,\C)$ acting on
    $\bigwedge^2_0 \C^{2m}$} for $m\ge 3$. Here $n = 2m^2-m-1 =p+q$.
  The dimension bound \eqref{dimso} then gives
  $2m^2+m\ge pq-p=p(2m^2-m-2-p)$. Already for $p=2$ this implies
  $m\le 2$, and hence there are no solutions.

  \bigskip\noindent{ \it $\bullet$ $H_\C = \Spin(m,\C)$ acting on a
    spin representation, $m = 7,9$}. Since the representation spaces
  are $\C^8$ and $\C^{16}$ respectively, the dimension bound leaves
  the following possibilities:
  \begin{itemize}
  \item $G = \SO_0(2,6)$, $\SO_0(3,5)$ or $\SO_0(4,4)$ if $m=7$,
  \item $G = \SO_0(2,14)$ or $ \SO_0(3,13)$ if $m=9$.
  \end{itemize}

  It follows from the signature laws of the spin representations
  \cite{H}*{Theorems 13.1 and 13.8} that only symmetric signatures
  (i.e.~of the form $(p,p)$) can occur. Hence only $G = \SO_0(4,4)$ is
  possible.  In that case
  $\hf=\spin(3,4):=\spin(7,\C)\cap\so(4,4)$. It is symmetric by Lemma
  \ref{real symmetric}, since $\spin(7,\C)$ is symmetric in
  $\so(8,\C)$.

  \bigskip\noindent{\it $\bullet$ $H_\C$ of exceptional type.} The
  case $H_\C =\sG_2^\C$ is possible; with $\hf=\sG_2^1$, the pair
  $(\so(3,4),\hf)$ is absolutely spherical (see Table
  \ref{non-symmetric real forms}). In view of
  Table~\ref{tab:prehomogeneous} we are left with
  $H_\C=\sF_4^\C$. Then $\dim V=26$ and the dimension bound implies
  that $G=\SO_0(2,24)$. The only non-compact real form of $H_\C$ with
  rank $\le 2$ is $\sF_4^2$. Its representation space
  \[ V = \{ X \in \Herm(3,\mathbb{O})_\C: TrX = 0\}. \] According to
  (2.2) in \cite{C}, the space $V$ carries an invariant symmetric
  bilinear form with signature $(10, 16)$. This is different from
  $(2,24)$.
\end{proof}

\subsubsection{The case $G=\SO^*(2m)$}

\begin{lemma}\label{typeIII-so*} Let $G=\SO^*(2m)$ for $m \geq 5$. 
  Then there exists no real spherical subgroup $H\subset G$ of type
  III.
\end{lemma}

\begin{proof} By assumption $m\ge 2$ and hence
  \[ m^2 - \frac{3}{2}m = \frac{m}{2} \cdot (2m - 3) > 2m =
    \dim(V).  \] Hence, if $H$ is spherical it follows from
  \eqref{dimsoev} that $\dim H > \dim V$. In particular, $V$ is then a
  representation from Proposition~\ref{prop:prehomogeneous}, to which
  also Remark \ref{remark on forms} applies.

  \bigskip\noindent{ \it $\bullet$ $H_\C$ simple, adjoint
    representation.}  Since $\dim H = \dim V$, this is impossible by
  the strictness of the dimension bound.

  \bigskip \noindent{ \it $\bullet$ $H_\C =\SO(k,\C)$ acting on
    $V = \C^{k}$}. This is possible for $k = 2m$, but then we would
  have $H_\C = G_\C$.

  \bigskip\noindent{ \it $\bullet$ $H_\C = \Spin(k,\C)$ acting on a
    spin representation for $k=7,9$}. Note that $k=7$ is excluded
  since $m\ge 5$. For $\Spin(9,\C)$ on $\C^{16}$ we have $m = 8$ and
  $\dim H = 36 < 52$, so $H$ does not satisfy the dimension bound.

  \bigskip\noindent{ \it $\bullet$ $H_\C = \Sp(k,\C)$ acting on
    $\bigwedge^2_0 \C^{2k}$}. Here $2m = 2k^2-k-1$.  Since $k \geq 3$,
  we have $m \geq 7$. Hence, it follows from
  \[ \dim H = 2k^2 + k = 2m + 2k + 1 \geq m(m - \frac{3}{2}) \] and
  $k\leq m$ that $4m + 1 \geq \frac{11}{2}m$, which is impossible.

  \bigskip\noindent{ \it $\bullet$ $H_\C$ of exceptional type.}  Only
  for $\sF_4^\C$ is $\dim V$ even and not excluded by
  Table~\ref{tab:prehomogeneous}.  Then $m=13$ and \eqref{dimsoev} is
  invalid.
\end{proof}

This concludes the proof of Theorem \ref{maximals are real forms} for
$G_\C=\SO(n,\C)$.

\section{Maximal reductive real spherical subgroups for the
symplectic  groups}\label{max Sp}

We only consider the real forms $G=\Sp(p,q)$ of $G_\C=\Sp(n, \C)$,
$p+q=n$ and $0<p\leq q$, as the real form $\Sp(n,\R)$ is split. Then
$\dim(G/K)=4pq$ and $\rank_\R G=p$, so that by \eqref{db}
\begin{equation}\label{dimsp} \dim H \geq 4pq - p . \end{equation}

\subsection{About $\Sp(p,q)$}
For later reference we record some structural facts for the group
$\Sp(p,q)$. As before we identify $\Hb^n$ with $\C^{2n}$ and denote by
$h\mapsto \oline h$ the conjugation on $\Hb^n$. The group $\Sp(p,q)$
consists of the right $\Hb$-linear transformations on $\Hb^n$ which
preserve the Hermitian $\Hb$-valued form
\[ \phi (h, h')= \oline{h_1} h_1' + \ldots + \oline{h_p} h_p' -
  \oline{h_{p+1}} h_{p+1}' -\ldots - \oline{h_n}h_n' \, .\]

Similar to the $\SO^*(2m)$-case the $\C$-part of $\phi$ yields a
Hermitian form $(\cdot, \cdot)_{2p,2q}$ and the $j\C$-part a
symplectic form $\la\cdot , \cdot\ra$, both being kept invariant under
$\Sp(p,q)$.  In particular, if we view $\Sp(p,q)$ as a subgroup of
$\SL(2n,\C)$, then $\Sp(p,q)= \Sp(n,\C) \cap \SU(2p,2q)$.

The minimal flag variety is given by the isotropic right $\Hb$-flags:
\begin{equation} \label{H-flag II} G/P =\{ V_1 \subset \ldots \subset
  V_{p} \subset\Hb^{n}\mid \dim_\Hb V_{i}=i,\, \phi(V_i, V_i)=\{0\} \}
  \, .\end{equation}

\begin{lemma}\label{codim7} Let $H\subset\Sp(p,q)$ be a real spherical subgroup. 
  Then $H$ admits an orbit on $\Pb(\Hb^n)$ of real codimension at most
  $1$ and an orbit on $\Hb^n=\C^{2p+2q}$ of real codimension at most
  $5$.
\end{lemma}

\begin{proof} Let ${\mathcal L}$ be the variety of $\phi$-isotropic
  $\Hb$-lines. According to \eqref{H-flag II} ${\mathcal L}$ is a
  $G$-quotient of $G/P$ and hence a real spherical subgroup
  $H\subset G$ must admit an open orbit on ${\mathcal L}$.  Observe
  that a line $v\Hb$ is isotropic if and only if the real valued
  function $v \mapsto \phi(v,v)$ vanishes.  From that the assertion
  follows.
\end{proof}

\subsection{Type I maximal subgroups}

Let $H\subset G$ be a maximal subgroup of type I. Then
$H_\C=\Sp(r,\C) \times \Sp(s,\C))$, $r,s>0$ and $n=r+s$ or
$H_\C=\GL(n, \C)$.  In both cases $H_\C$ is a symmetric subgroup of
$G_\C$. Hence with Lemma \ref{real symmetric} and Berger's list
\cite{Berger} we obtain:

\begin{lemma} \label{typeI-sp}Let $H\subset\Sp(p,q)$ be a subgroup of
  type I. Then $H$ is symmetric and up to conjugation one of the
  following:
  \begin{enumerate}
  \item $\Sp(p_1,q_1) \times \Sp(p_2,q_2)$ with $p_1+p_2=p$ and
    $q_1+q_2=q$,
  \item $\Sp(p,\C)$ if $p=q$,
  \item $\operatorname{U}(p,q)$,
  \item $\GL(p,\Hb)$ for $p=q$.
  \end{enumerate}
\end{lemma}

\subsection{Type II maximal subgroups}

In this situation we have $H_\C = \Sp(r, \C) \otimes \SO(s,\C)$ with
$s\geq 3$, $s\neq 4$ or $(r,s)=(1,4)$.

\begin{lemma}\label{typeII-sp} There are no real spherical subgroups
  $H\subset\Sp(p,q)$ of type II.
\end{lemma}

\begin{proof} We first claim that the real forms
  $H_1=\Sp(r, \R) \otimes \SO_0(p_2, q_2)$ with $p_2+q_2=s$ and
  $H_2=\Sp(r,\R)\otimes \SO^*(2m)$ with $2m=s$ are not possible.  Note
  that $\Sp(r,\R)\subset \SU(r,r)$,
  $\SO_0(p_2,q_2)\subset \SU(p_2,q_2)$, $\SO^*(2m)\subset \SU(m, m)$.
  It follows that both $H_1$ and $H_2$ leave a Hermitian form on
  $\C^n=\C^{2r}\otimes\C^s$ invariant which is of type $(rs,rs)$ and
  hence $p=q=\frac12rs$ by Lemma \ref{signature lemma}.  The dimension
  bound \eqref{dimsp} then gives
  \[ 2r^2 +r + \frac12s(s-1)\geq 4p^2-p\] with $rs=2p$. This has no
  solutions since for $r\ge 1$, $s\ge 3$ we find
  \[ 2r^2 +r + \frac12 s^2\le \frac29 r^2s^2+ \frac 19 r^2s^2+\frac12
    r^2s^2=\frac56 r^2s^2\le r^2s^2-\frac12rs\] and thus neither $H_1$
  nor $H_2$ can be spherical.

  The case where $H=\Sp(p_1,q_1)\otimes \SO^*(2m)$ is similar, as $H$
  leaves a Hermitian form of equal parity invariant.

  This leaves us with the last case where
  $H=\Sp(p_1,q_1) \otimes \SO_0(p_2,q_2)$. It requires a more detailed
  investigation.  We request $p_1\le q_1$ and $q_2\le p_2$.  Then
  $p_1p_2 + q_1 q_2\le p_1 q_2 + p_2q_1$ and
  \[( 2p_1p_2 + 2q_1 q_2, 2p_1 q_2 + 2p_2q_1)=(2p, 2q)\] as $H$ leaves
  invariant a Hermitian form of this signature (cf. Lemma
  \ref{signature lemma}).  Inserting that in the dimension bound
  \eqref{dimsp} gives
  \[ 2r^2+r+\frac12 s(s-1) \ge
    4p_2q_2(p_1^2+q_1^2)+4p_1q_1(p_2^2+q_2^2) -(p_1p_2+q_1q_2).\]

  As in \eqref{dimbd} we deduce that one factor must be compact.
  Suppose first that $q_2=0$ hence $H=\Sp(p_1,q_1) \otimes \SO(s)$
  with $s=p_2$.  The dimension bound in this case is:
  \[ 2r^2 +r + \frac12 {s(s-1)} \geq 4 q_1p_1 s^2 - p_1 s\,.\] There
  are no solutions for $2r\leq s$. For $s\leq 2r$, a matrix
  computation (use an analogue of the map \eqref{matrix submersion})
  combined with Lemma \ref{codim7} yields that $\SO(s)$ needs to have
  an orbit of real codimension at most $5$ on $\Herm(s,\C)$. The
  orbits of maximal dimension are in $\Sym(s,\R)$, and they have
  codimension $s$ in this space, hence $\frac12s(s+1)$ in
  $\Herm(s,\C)$. It follows that no $s\geq 3$ meets the requirement.

  Finally we investigate the case where $p_1=0$. Then $r=q_1$ and
  $H=\Sp(r) \otimes \SO_0(p_2,q_2)$ and the dimension inequality
  becomes:
  \[ 2r^2 +r + \frac12{s(s-1)} \geq 4 q_2p_2 r^2 - p_2r \] There is no
  solution if $3\le s\leq 2r$ so we may assume that $s> 2r$.  With the
  matrix computations similar to the $\SU(p,q)$-case (see \eqref{matrix
    submersion}) combined with Lemma \ref{codim7} this reduces
  matters to study $\Sp(r)$-orbits on $\Herm(2r, \C)$ with
  codimension at most $5$.  This implies $r=1$,
  i.e.~$H=\Sp(1)\otimes\SO_0(p,q)$ with $p+q\ge 3$.  We now proceed as
  in Lemma \ref{typeII-su}: Consider
  $H':= \1 \otimes\SO_0(p,q) \subset H$.  Then $H'$ is required to
  have an orbit on $G/P$ of codimension at most $3$.  We now produce
  many $H'$-invariant functions on $G/P$.  First we decompose
  $V=\C^n = \C^{p +q} \oplus \C^{p+q}$ into $H'$-orthogonal summands
  and write $p_i: V \to \C^{p+q}$, $1\leq i\leq 2$ for the two
  $H'$-equivariant projections.  Now given a flag
  $\F=\{V_1 \subset\ldots\subset V_{p}\}$, we choose an orthonormal
  basis $v_1, \ldots, v_{2p}$ of $V_{p}$ with respect to the standard
  Hermitian inner product on $\C^n$ such that $V_{i}$ is spanned by
  $v_1, \ldots, v_{2i}$.  Denote by $(\cdot,\cdot)_{p,q}$ the complex
  bilinear form on $\C^{p+q}$ which is invariant for $H'$.

  Then for $1\leq m\leq p$ and $1\leq j, k\leq 2$ we consider the
  function
  \[ g_{mjk}(\F):= (p_j(v_1\wedge \ldots\wedge v_{2m}),
    p_k(v_1\wedge\ldots\wedge v_{2m}))_{p,q}\, .\] Similarly as before
  the rational functions
  \[ f_{mjk}:= \re \left(\frac {g_{mjk}}{g_{m11}} \right)\quad
    \hbox{and}\quad f_{mjk}':= \im
    \left(\frac{g_{mjk}}{g_{m11}}\right)\] are all $H'$-invariant.
  Already for $m=1$ we obtain 4 independent invariants this way, and
  thus $H'$ cannot have an orbit of codimension 3 by Lemma
  \ref{independent}.
\end{proof}

\subsection{Type III maximal subgroups}

\begin{lemma}\label{typeIII-sp} Let $G=\Sp(p,q)$. Then there exist no
  real spherical subgroups of type III.
\end{lemma}

\begin{proof} We may assume that $1<p$ as it is known for $p=1$ by
  \cite{KS1}.  Then $2\le p\le q$ implies $3p+2q\le5q<8q\le4pq$ and
  hence \[ 4pq - p > \dim V = 2p + 2q. \] Hence we get from
  \eqref{dimsp} the strict inequality $\dim H_\C >\dim V$, and again
  we can use Proposition~\ref{prop:prehomogeneous} and Remark
  \ref{remark on forms}.  We are thus left with testing some sporadic
  cases, and it is easy to see that they never satisfy the dimension
  bound.
\end{proof}

This concludes the proof of Theorem \ref{maximals are real forms} for
$G_\C=\Sp(n,\C)$.

\section{The maximal real spherical subalgebras of the exceptional Lie
  algebras}\label{max ex}

Here $\gf$ is such that $\gf_\C$ is exceptional simple.  We assume
that $\gf$ is not compact.

\begin{lemma}\label{max is symmetric} 
  Let $\gf$ be a non-complex exceptional non-compact simple real Lie
  algebra and $\hf$ a real spherical maximal reductive subalgebra.
  Then,

  \begin{enumerate}
  \item If $\gf\neq \sG_2^1$, then $\hf$ is symmetric.
  \item If $\gf=\sG_2^1$, then $\hf$ is symmetric or conjugate to
    either $\hf_1=\su(2,1)$ or $\hf_2=\sl(3,\R)$ which are both
    absolutely spherical but not symmetric.
  \end{enumerate}
\end{lemma}

\begin{proof} Recall from Corollary \ref{Komrakov cor} that $\hf_\C$
  is maximal reductive in $\gf_\C$.  If $\gf$ is quasi-split, then the
  lemma follows from Lemma \ref{lemma quasi-split} combined with the
  work of Krämer \cite{Kr} (see Table \ref{non-symmetric spherical}).
  In particular $\sG_2^1$, the only non-compact real form of
  $\sG_2^\C$, is split, and thus the assertion (2) is obtained with
  Table \ref{non-symmetric real forms}.

  {}From now on we assume that $\gf$ is not quasi-split. For
  $\gf_\C=\sF_4^\C$ the only non-split real form $\sF_4^2$ has real
  rank one, and for that the result is given in \cite{KS1}*{Lemma 6.2}.
  This leaves us to consider for $\gf$ only the real forms $\sE_6^3$,
  $\sE_6^4$ of $\sE_6^\C$, $\sE_7^2$, $\sE_7^3$ of $\sE_7^\C$ and
  $\sE_8^2$ of $\sE_8^\C$.

  We follow \cite{Mik}.  According to Dynkin \cite{Dynk}, a subalgebra
  $\hf_\C$ of $\gf_\C$ is called {\it regular}, if it is normalized by
  a Cartan subalgebra of $\gf_\C$. On the other hand, $\hf_\C$ is
  called an {\it S-subalgebra} of $\gf_\C$ if it is not contained in
  any proper regular subalgebra of $\gf_\C$.

  Let $\hf_\C$ be a maximal reductive subalgebra of $\gf_\C$. Then it
  is either regular or an S-subalgebra. According to
  \cite{Dynk}*{Theorem 14.1}, the pairs $(\gf_\C: \hf_\C)$, where
  $\hf_\C$ is non-symmetric and a maximal S-subalgebra of $\sE_6^\C$,
  $\sE_7^\C$ or $\sE_8^\C$, are given by:
  \begin{align*}
    &(\sE_6^\C: ~ \sA_1, ~ \sG_2^\C, ~ \sA_2 \oplus \sG_2^\C),\\ 
    &(\sE_7^\C: ~ \sA_1, \sA_1 \oplus \sA_1,~ \sA_2, ~ \sG_2^\C \oplus \sC_3, ~ 
      \sA_1 \oplus \sF_4^\C, ~ \sA_1 \oplus \sG_2^\C),\\
    &(\sE_8^\C: ~ \sA_1, ~ \sA_1 \oplus \sA_2, ~ \sB_2, ~ \sG_2^\C \oplus \sF_4^\C),
  \end{align*}
  and by \cite{Dynk}*{Theorem 5.5} (together with the correction on
  p. 311 of the selected works) the pairs $(\gf_\C: \hf_\C)$, where
  $\hf_\C$ is non-symmetric, semisimple, and a maximal regular
  subalgebra, are given by:
  \begin{align*}
    &(\sE_6^\C: ~ \sA_2 \oplus \sA_2 \oplus \sA_2), \\
    &(\sE_7^\C: ~ \sA_2 \oplus \sA_5), \\
    &(\sE_8^\C: ~ \sA_8, ~ \sA_4 \oplus \sA_4, ~\sA_2 \oplus \sE_6^\C).
  \end{align*}

  Note that a maximal reductive subalgebra of $\gf_\C$ is either a
  semisimple maximal subalgebra or a maximal Levi subalgebra of
  $\gf_\C$.  From the Dynkin diagram of $\gf$ we can read the maximal
  Levi subalgebras and deduce that they are either symmetric or are
  contained in a semisimple maximal regular subalgebra listed above
  (see \cite{BdS}*{Table on p.\ 219}).  Hence the two lists together
  consist in fact of all maximal reductive subalgebras which are not
  symmetric.

  Next we record the dimension bounds obtained from \eqref{db}:
  \begin{align}
    \label{dimE_6^3} \dim H &\geq 30 \qquad  & (\gf=\sE_6^3)\\ 
    \label{dimE_6^4}  \dim H &\geq 24 \qquad &(\gf=\sE_6^4)\\
    \label{dimE_7^2}  \dim H &\geq 60 \qquad &(\gf=\sE_7^2)\\
    \label{dimE_7^3} \dim H &\geq 51 \qquad  &(\gf=\sE_7^3)\\ 
    \label{dimE_8^2}  \dim H &\geq 108  \qquad &(\gf=\sE_8^2)
  \end{align}

  Going through the lists of maximal regular- and S-subalgebras of
  $\gf_\C$, we see that only the following two pairs $(G,H_\C)$
  satisfy the bound and thus may correspond to real spherical pairs:
  \[(\sE_6^4, \sA_2 \oplus \sA_2 \oplus \sA_2)\qquad\hbox{and} \qquad
    (\sE_7^3, \sA_1\oplus \sF_4^\C)\, .\]

  We claim that $G=\sE_6^4$ and a real form in $G$ of
  $H_\C=\SL(3,\C) \times\SL(3,\C) \times\SL(3,\C)$ cannot correspond
  to a real spherical pair.  By inspecting the Satake diagram of
  $\sE_6^4$ we see that the minimal parabolic $P$ of $G$ is contained
  in the maximal parabolic $P_{{\rm max},\C}$ of $G_\C$, which is
  related to the $27$-dimensional fundamental representation of
  $G_\C$. This representation is prehomogeneous, see case (\Esix) in
  Table \ref{tab:prehomogeneous} of Proposition
  \ref{prop:prehomogeneous}.  If the pair was spherical then $\C^{27}$
  would thus become a prehomogeneous vector space for $H_\C$. As
  $\dim H_\C=24<26$ this is excluded.

  This leaves us with the case $(\sE_7^3, \sA_1\oplus \sF_4^\C)$. An
  inspection of the Satake diagram of $G=\sE_7^3$ shows that the
  minimal parabolic $P$ of $G$ is contained in the maximal parabolic
  $P_{{\rm max},\C}$ of $G_\C$, which is related to the
  $56$-dimensional fundamental representation of $G_\C$. Again this is
  prehomogeneous, see case (\Eseven) in Table
  \ref{tab:prehomogeneous}.  If the pair was spherical then $\C^{56}$
  would thus be a prehomogeneous vector space for
  $H_\C = \SL(2,\C) \times \sF_4^\C $.  In particular, every
  irreducible $H_\C$-submodule of $\C^{56}$ is then prehomogeneous.
  We recall from \cite{SK}*{Thm.~54}, that $\sF_4^\C$ does not admit a
  prehomogeneous vector space in the generalized sense: for no
  non-trivial irreducible representation $V$ of $\sF_4^\C$ and for no
  $n\in \N$ does $\sF_4^\C\times \GL(n,\C)$ admit an open orbit on
  $V \otimes \C^n$.  In particular, $\C^{56}$ cannot be prehomogeneous
  for $H_\C$.
\end{proof}

The lemmas in Sections \ref{max SL}-\ref{max ex} together with the
list of Krämer \cite{Kr} finally conclude the proofs of Theorem
\ref{maximals are real forms} and Lemma \ref{most maximals are
  symmetric}.

\section{Tables for $L\cap H$}\label{section:LcapH}

Let $(\gf,\hf)$ be a real spherical pair and recall from Section
\ref{Towers} the parabolic subgroup $Q=L \ltimes U$ adapted to $Z=G/H$
and $P$.  In view of Proposition \ref{tower-factor} the Lie algebra
$\lf\cap \hf$ is of central importance to us.  In the following Tables
\ref{lcaph1}-\ref{lcaph2} (classical and exceptional) we list all
symmetric pairs $(\gf,\hf)$ (from Berger's list \cite{Berger}) with
$\gf$ not quasi-split nor compact, together with the associated
subalgebra $\lf\cap\hf$.

\begin{table}[ht]
  {\small 
    \[\begin{array}{llllll}
        &\gf&\hf&\lf\cap\hf\\
        \hline
        \zz&\su(p,q)&\so(p,q)&\so(q-p)&1\le p\le q-2\\
        \zz&\sl(n,\HH)&\so^*(2n)&\uf(1)^n&n\ge3\\
        \zz&\su(2p,2q)&\sp(p,q)&\sp(q-p)+\sl(2,\C)^p&1\le p\le q-1\\
        \zz&\sl(n,\HH)&\sp(n-k,k)&\sp(1)^n&0\le k\le\frac n2\\

        \zz&\su(p,q)&\sf[\uf(p_1,q_1)+\uf(p_2,q_2)]&
                                                        \begin{Cases}\sf[\uf(p_2-q_1,q_2-p_1)+\uf(1)^{p_1+q_1}]\\
                                                          \sf[\uf(q_1-p_2)+\uf(q_2-p_1)+\uf(1)^{p_1+p_2}]\end{Cases}
        &\!\!\begin{array}{l}p_2\ge q_1\\p_2\le q_1\end{array}\\
        (\text{6a})&\sl(n,\HH)&\sf[\gl(n-k,\HH)+\gl(k,\HH)]&\sf[\gl(n-2k,\HH)+\gl(1,\HH)^k]&0\le k\le\frac n2\\
        (\text{6b})&\sl(n,\HH)&\sl(n,\C)+\uf(1)&\begin{Cases}\sf[\gl(1,\HH)^{\frac n2}]\\\sf[\gl(1,\HH)^{[\frac n2]}]+\uf(1)\end{Cases}&\!\!\begin{array}{l}n\text{ even}\\n\text{ odd}\end{array}\\
        \hline \stepcounter{zeile}

        \zz&\so(2p,2q)&\uf(p,q)&\uf(q-p)+\su(1,1)^p&1\le p\le q-2\\
        (\text{8a})&\so^*(2n)&\uf(n-k,k)&
                                          \begin{Cases}\su(2)^{\frac n2},~n\text{ and }k\text{ even}\\
                                            \su(2)^{\frac n2-1}+\su(1,1)+\uf(1),~n\text{ even},k\text{ odd }\\
                                            \su(2)^{\frac{n-1}2}+\uf(1),~n\text{
                                              odd}
                                          \end{Cases}
        &0\leq k\le \frac n2,~n\ge4\\
        (\text{8b})&\so^*(2n)&\gl({\textstyle\frac n2},\HH)&\su(2)^{\frac n2}&n\ge4\text{ even}\\
        \stepcounter{zeile}

        \zz&\so(p,q)&\so(p_1,q_1)+\so(p_2,q_2)&\begin{Cases}\so(p_2-q_1,q_2-p_1)\\\so(q_1-p_2)+\so(q_2-p_1)\end{Cases}&\!\!\begin{array}{l}p_2\ge q_1\\p_2\le q_1\end{array}\\
        (\text{10a})&\so^*(2n)&\so^*(2n-2k)+\so^*(2k)&\so^*(2n-4k)+\uf(1)^k&1\le k\le\frac n2\\
        (\text{10b})&\so^*(2n)&\so(n,\C)&\uf(1)^{[\frac n2]}&n\ge 4\\\hline
        \stepcounter{zeile}

        \zz&\sp(p,q)&\begin{Cases}\uf(p,q)\\\gl(p,\HH),~p=q\end{Cases}&\uf(q-p)+\uf(1)^p&1\le p\le q\\
        (\text{12a})&\sp(p,q)&\sp(p_1,q_1)+\sp(p_2,q_2)&\begin{Cases}\sp(p_2-q_1,q_2-p_1)+\sp(1)^{p_1+q_1}\\\sp(q_1-p_2)+\sp(q_2-p_1)+\sp(1)^{p_1+p_2}\end{Cases}&\!\!\begin{array}{l}p_2\ge q_1\\p_2\le q_1\end{array}\\
        (\text{12b})&\sp(p,p)&\sp(p,\C)&\sp(1)^p&\\\hline
      \end{array}\]
    \smallskip
  }
  \centerline{\rm\Tabelle{lcaph1}}
\end{table}

\smallskip

The notation in Table~\ref{lcaph1} follows certain conventions: in
each row where the letters appear one has
$p=p_1 + p_2\leq q=q_1 + q_2$ and $p_1 + q_1 \leq p_2 + q_2$ (whence
$p_1\le q_2$).

Following are some remarks on how the intersections $\lf\cap\hf$ have
been calculated. For this we made extensive use of \cite{KK},
especially \S10. The complexification $Z_\C=G_\C/H_\C$ of the
symmetric space $Z=G/H$ is (complex) spherical.  We may assume that
$Z_\C$ admits a wonderful compactification, see Remark \ref{remark
  curlyQ}. Its Luna diagram is a collection of various data of which
we need only two. First, a subset $S^{(p)}$ of the set $S$ of simple
roots of $G$ whose elements are called the \emph{parabolic roots of
  $Z$}. Second, a finite set $\Sigma_Z$ of characters of $G$, called
the \emph{spherical roots of $Z$}. Each spherical root is an
$\N$-linear combination of simple roots.

The real structure provides us with the set $S^0\subset S$ of compact
simple roots (the black dots in the Satake diagram). Then, as
mentioned already in Remark \ref{remark curlyQ} , the set of simple
roots of $L$ is the union $S^0\cup S^{(p)}$. Now let
$\Sigma_Z^0\subset \Sigma_Z$ be the set of those spherical roots which
lie in the span of $S^0\cup S^{(p)}$. Then \cite{KK}*{Cor.~10.16}
implies that $Z^0:=L/L\cap H$ is an absolutely spherical variety whose
Luna diagram has still $S^{(p)}$ as set of parabolic roots and
$\Sigma_Z^0$ as set of spherical roots. Since $L\cap H$ is reductive,
these two data suffice to determine the isomorphism type of the
derived subgroup $(L\cap H)'$ by use of tables in \cite{BP}.

To determine the connected center $C$ of $L\cap H$ it suffices to know
its dimension and its real rank. The local structure theorem implies
that $L/L\cap H$ is an open subset of the double coset space
$U\backslash G/H$ where $U$ is the unipotent radical of the adapted
parabolic of $Z$. From this we get
\[
  \dim L\cap H=\dim H-\dim U=\dim H-\frac12(\dim G-\dim L).
\]
Knowing $(L\cap H)'$ we get $\dim C$. For its real rank, we use
\[
  \rank_\R C=\rank_\R L-\rank_\R Z=\rank_\R
  L-\dim\langle\text{res}_A\sigma\mid\sigma\in\Sigma_Z\rangle_\Q.
\]
(see \cite{KK}*{Lemma~4.18}). Here $A\subset L$ is a maximally split
subtorus.

\begin{table}[!h]
  {\small 
    \[\begin{array}{lll}
        \gf&\hf&\lf\cap\hf\\
        \hline
        \sE_6^3&\sp(2,2)&\so(4)\\
        \sE_6^4&\sp(3,1)&\so(4)+\so(4)\\
        \hline
        \sE_6^3&\begin{cases}\su(4,2)+\su(2)\\\su(5,1)+\sl(2,\R)\end{cases}&\uf(2)+\uf(2)\\
        \sE_6^4&\sl(3,\HH)+\su(2)&\so(5)+\so(3)+\gl(1,\R)\\
        \hline
        \sE_6^3&\begin{cases}\so(10)+\uf(1)\\\so(8,2)+\uf(1)\\\so^*(10)+\uf(1)\end{cases}&\uf(4)\\
        \sE_6^4&\so(9,1)+\gl(1,\R)&\spin(7)+\gl(1,\R)\\
        \hline
        \sE_6^3&\sF_4^2&\so(7,1)\\
        \sE_6^4&\begin{cases}\sF_4^2\\\sF_4\end{cases}&\so(8)\\
        \hline
        \sE_7^2&\begin{cases}\su(6,2)\\\su(4,4)\end{cases}&\so(2)^3\\
        \sE_7^3&\begin{cases}\su(6,2)\\\sl(4,\HH)\end{cases}&\so(4)+\so(4)\\
        \hline
        \sE_7^2&\begin{cases}\so(12)+\su(2)\\\so(8,4)+\su(2)\\\so^*(12)+\sl(2,\R)\end{cases}&\su(2)^3\\
        \sE_7^3&\begin{cases}\so(10,2)+\sl(2,\R)\\\so^*(12)+\su(2)\end{cases}&\so(6)+\so(2)+\sl(2,\R)\\
        \hline
        \sE_7^2&\begin{cases}\sE_6^2+\uf(1)\\\sE_6^3+\uf(1)\end{cases}&\so(6,2)+\so(2)\\
        \sE_7^3&\begin{cases}\sE_6^3+\uf(1)\\\sE_6^4+\gl(1,\R)\\\sE_6+\uf(1)\end{cases}&\so(8)\\
        \hline
        \sE_8^2&\begin{cases}\so(12,4)\\\so^*(16)\end{cases}&\so(4)+\so(4)\\
        \hline
        \sE_8^2&\begin{cases}\sE_7^2+\su(2)\\\sE_7^3+\sl(2,\R)\\\sE_7+\su(2)\end{cases}&\so(8)\\
        \hline
        \sF_4^2&\sp(2,1)+\su(2)&\so(4)+\so(3)\\
        \hline
        \sF_4^2&\begin{cases}\so(9)\\\so(8,1)\end{cases}&\spin(7)\\
      \end{array}\]
    \smallskip
  }
  \centerline{\Tabelle{lcaph2}}
\end{table}

In most cases, it is not necessary to know the embedding of
$\lf\cap\hf$ into $\hf$ but in some it does matter. For this, the
following lemma is useful.

\begin{lemma}\label{lcaph}
  Let $\hf_1$, $\hf_2$ be two self-normalizing real spherical
  subalgebras of $\gf$ with adapted parabolic subalgebras $\qf_1$ and
  $\qf_2$ corresponding to minimal parabolic subalgebras $\pf_1$ and
  $\pf_2$.  Suppose that $\hf_{1,\C} =\Ad(x) \hf_{2,\C}$ for some
  $x\in G_\C$.  Then there exists an element $g\in G_\C$ of the form
  $g=tg_0$ with $g_0\in G$ and $t\in Z(L_{2,\C})$ such that $\Ad(g)$
  maps $\hf_{1,\C}$, $\qf_{1,\C}$, and $\lf_1\cap\hf_1$ onto
  $\hf_{2,\C}$, $\qf_{2,\C}$, and $\lf_2\cap\hf_2$, respectively.
\end{lemma}

\begin{proof} See \cite{KK}*{Section 13}.
\end{proof}

In particular, the lemma says that the complexification of the
embedding $\lf\cap\hf\hookrightarrow\hf$ does not depend on the
particular real form $\hf$. This is used in part (a) of the following
remark.

\begin{remark}\label{lcaph embedding}
  (a) In Table \ref{lcaph2} there is ambiguity how $\lf\cap \hf$ is
  embedded into $\hf$ in some cases where $\hf=\hf_1 \oplus \hf_2$
  consists of two factors.  However, with Lemma \ref{lcaph} one can
  derive the following additional data from the table:
  \begin{itemize}
  \item[($\rm a_1$)] For $\gf=\sE_6^3$ and
    $\hf_\C=\sl(6,\C) \oplus \sl(2,\C)$ one has
    $[\lf\cap \hf, \lf\cap \hf]\subset \hf_1$.
  \item[($\rm a_2$)] For $\gf=\sE_7^2$ and
    $\hf_\C=\so(12,\C) \oplus \sl(2,\C)$ one has
    $\lf \cap \hf\subset \hf_1$.
  \item[($\rm a_3$)] For $\gf=\sE_7^3$ and
    $\hf_\C=\so(12,\C) \oplus \sl(2,\C)$, one has
    $[\lf\cap \hf, \lf\cap \hf]\subset \hf_1$.
  \end{itemize}

  To see that, we discuss the case ($\rm a_1$). The arguments for
  ($\rm a_2$) and ($\rm a_3$) are similar.  Table \ref{lcaph2} shows
  that there are two symmetric subalgebras in $\gf$, say
  $\hf'=\hf_1' \oplus \hf_2'=\su(4,2)\oplus\su(2)$ and
  $\hf''=\hf_1'' \oplus\hf_2''=\su(5,1)\oplus\sl(2,\R)$ with
  isomorphic complexifications.  By Lemma \ref{lcaph} there is an
  isomorphism of $\gf_\C$ which carries $\hf''_\C$ to $\hf'_\C$ and
  $\lf''\cap\hf''$ to $\lf'\cap\hf'$.  Table \ref{lcaph2} shows that
  $\lf\cap\hf$ is of compact type, and hence
  $[\lf''\cap\hf'',\lf''\cap\hf'']\subset \hf_1''$ by Schur's lemma.
  It then follows that $[\lf'\cap\hf',\lf'\cap\hf']\subset \hf_1'$.

  (b) For $\gf=\sF_4^2$ and $\hf=\sp(2,1) \oplus\su(2)$ the algebra
  $\lf\cap \hf$ surjects onto $\hf_2$.  To see this, let $V$ be the
  $52$-dimensional irreducible (adjoint) representation of
  $\sF_4$. Then we claim that $\dim V^{\lf\cap\hf}=1$. This can be
  shown by branching $V$ (with highest weight $\omega_1$) to
  $\lf\cap\hf$ by using the following chain of maximal subalgebras
  \[
    \lf\cap\hf=\so(4)+\so(3)\subset\lf'=\so(7)\subset\so(8)\subset\so(9)\subset\sF_4.
  \]
  One can do that either by hand (starting with
  $\Res_{\so(9)}^{\sF_4}V=L(\omega_2)+L(\omega_4)$) or by help of a
  computer algebra package. We used \textsf{LiE}, \cite{LiE}, with the
  functions \texttt{resmat()} to generate the restriction matrices and
  \texttt{branch()} to perform the branching.

  On the other hand, $\mathrm{res}^\gf_\hf V$ contains the
  $3$-dimensional $\sp(3)\oplus\sl(2)$-module
  $\C\otimes S^2\C^2$. This cannot happen if the projection of
  $\lf\cap\hf$ to $\hf_2$ were trivial.

  (c) In the classical case (Table \ref{lcaph1}) there are also some
  situations where $\hf$ is not simple, and where it is of interest
  how certain factors of $\lf\cap\hf$ are embedded into $\hf$.  These
  are:

  \begin{itemize}
  \item In (5) $\uf(1)^{p_1+q_1}$, resp. $\uf(1)^{p_1+p_2}$, is
    diagonally embedded into $\hf=\hf_1 +\hf_2$,
  \item In (6a) $\gl(1,\Hb)^k$ is diagonally embedded into
    $\hf=\hf_1 +\hf_2$,
  \item In (10a) $ \uf(1)^k$ is diagonally embedded into
    $\hf=\hf_1 +\hf_2$,
  \item In (12a) $\sp(1)^{p_1+q_1}$, resp. $\sp(1)^{p_1+p_2}$,
    is diagonally embedded into $\hf=\hf_1 +\hf_2$.
  \end{itemize}

  This can be verified as follows: Let $\sigma$ be the involution
  which determines $\hf$ and let $\theta$ be the standard Cartan
  involution which commutes with $\sigma$. Let $\kf$ be the fixed
  point set of $\theta$.  Then $\lf$ can be chosen as the centralizer
  of a generic element $X\in \hf^\perp \cap\kf^\perp$ where $\perp$
  refers to the orthogonal complement with respect to the
  Cartan-Killing form of $\gf$.  Simple matrix computations then
  verify the bulleted assertions.

  (d) In the last two lines of Table \ref{lcaph2} we have
  $\lf\cap\hf=\spin(7)$. That it is the spin embedding (and not
  $\so(7)$) is seen in both cases from the fact that the complement in
  $\gf$ contains the spin representation.

\end{remark}

\section{The classification of reductive real spherical pairs}

Now that we have classified all maximal spherical subalgebras which
are reductive, we can complete the classification.

We recall the adapted parabolic $Q=L\ltimes U\supset P$ of a real
spherical space.  We set $L_H:=L\cap H$ and denote its Lie algebra by
$\lf_\hf$. Further we may assume that $MA\subset L$.

The general strategy is as follows. Given $G$ and a maximal reductive
real spherical subgroup $H$ we let $H'\subset H$ be a proper reductive
subgroup.  According to Proposition \ref{tower-factor} the space
$G/H'$ is real spherical if and only if $H/H'$ is a real spherical
$L_H$-variety.  In particular,
\begin{equation} \label{factor condition} \hf = \hf'
  +\lf_\hf\end{equation} needs to hold by Corollary \ref{cor
  tower-factor}. By Lemma \ref{most maximals are symmetric}, $H$ is
symmetric in almost all cases, and hence $\lf_\hf$ is given by the
tables in Section \ref{section:LcapH}. By Proposition \ref{Oni2} we
can then determine whether \eqref{factor condition} is valid and thus
limit the number of subgroups $H'$ to consider.

After the following preliminary result this section will be divided
into two parts: classical and exceptional.

\subsection{Almost absolutely spherical pairs}
In addition to \eqref{factor condition} there is a second general fact
which will be useful in the classification.  Let us call $\hf$
\emph{almost absolutely spherical} if it is real spherical and there
exists an absolutely spherical subalgebra $\fhq$ of $\gf$ with
$[\fhq,\fhq]\subset \hf\subset \fhq$.

\begin{lemma} \label{center removed} Let $\gf$ be a non-compact and
  non-complex simple Lie algebra and $\hf$ a reductive subalgebra
  which is not absolutely spherical. Then $(\gf,\hf)$ is almost
  absolutely spherical if and only if it is isomorphic to one of the
  pairs in Table \ref{table real spherical} of Theorem \ref{theorem
    classification} which are marked by an asterisk.
\end{lemma}

\begin{proof}
  We use the real version of the Vinberg-Kimel'feld criterion (see
  \cite{KKS}*{Prop.~3.7}): the subalgebra $\hf$ is real spherical if and
  only if $\dim V^\hf\le 1$ for all simple representations of $G$, for
  which there exists a $P$-semiinvariant vector.  Observe, that it
  suffices to check this condition over $\C$.

  Now let $\fhq\subset \gf$ be an absolutely spherical subalgebra in
  which $\hf$ is coabelian. Because $\hf$ is not absolutely spherical
  the complexified pair $(\gf_\C,\hf_\C)$ will appear (according to
  Krämer \cite{Kr}) in the following table:
  \[
    \begin{array}{lllll}
      \gf_\C&\hf_\C&\overline\cM&\alpha\\
      \hline
      \sl(2n,\C)&\sl(n,\C)+\sl(n,\C)&\omega_1,\ldots,\omega_{n-1},\omega_n\pm\epsilon,\omega_{n+1},\ldots,\omega_{2n-1}&\alpha_n\\
      \so(4n,\C)&\sl(2n,\C)&\omega_2,\omega_4,\ldots,\omega_{2n-2},\omega_{2n}\pm\epsilon&\alpha_{2n}\\
      \so(2n+1,\C)&\sl(n,\C)&\omega_1,\ldots,\omega_{n-1},\omega_n\pm\epsilon&\alpha_n\\
      \so(n,\C)&\so(n-2,\C)&\omega_1\pm\epsilon,\omega_2&\alpha_1\\
      \so(10,\C)&\spin(7,\C)&\omega_1\pm2\epsilon,\omega_2,\omega_4+\epsilon,\omega_5-\epsilon&\alpha_1\\
      \sp(n,\C)&\sl(n,\C)&2\omega_1,\ldots,2\omega_{n-1},\omega_n\pm\epsilon&\alpha_n\\
      \sp(n+1,\C)&\sp(n,\C)&\omega_1\pm\epsilon,\omega_2&\alpha_1\\
      \sE_7^\C&\sE_6^\C&\omega_1,\omega_6,\omega_7\pm\epsilon&\alpha_7\\
    \end{array}
  \]

  \noindent Since $\Hq_\C$ normalizes $X=G_\C/H_\C$, there is a right
  action of $T_0:=\Hq_\C/H_\C\cong\C^*$ on $X$. Moreover, because
  $G_\C/\Hq_\C$ is (absolutely) spherical, $X$ is spherical as
  $\Gq=G_\C\times T_0$-variety. The corresponding weights
  (i.e.~highest weights of irreducible $\Gq$-modules $V$ containing a
  non-trivial $\Hq_\C$-fixed vector) are called the extended weights of
  $X$. They are characters of $B\times T_0$, where
  $B\subset P_\C\subset G_\C$ is a Borel subgroup, and form a monoid
  $\overline\cM$. Now the third column shows the generators of this
  monoid. Here, $\epsilon$ generates the character group of $T_0$. The
  expansion of a character is given w.r.t.~the fundamental weights
  following the Bourbaki notation. The set of weights $\cM$ of $X$ as
  a $G$-variety is obtained by dropping the $T_0$-components, i.e., by
  setting $\epsilon=0$. This way we get a surjective map
  $\pi:\overline\cM\to\cM$. Let $\overline\cM_P\subset \overline\cM$
  be the submonoid of weights whose first component (i.e.~its
  restriction to $B$) is a weight of $P_\C$. Then the
  Vinberg-Kimel'feld criterion implies that $G/H$ is real spherical if
  and only if the restriction of $\pi$ to $\overline\cM_P$ is
  injective.

  To decide injectivity one checks that in every case there is a
  unique simple root $\alpha$ of $\gf_\C$ (given in the fourth column)
  with the property that the restriction of $\pi$ to
  $\overline\cM\cap H_\alpha$ is injective where $H_\alpha$ is the
  hyperplane perpendicular to $\alpha$.  We claim that $G/H$ is real
  spherical if and only if $\alpha$ is a compact simple root of
  $G$. Indeed, if $\alpha$ is compact then
  $\overline\cM_P\subset H_\alpha$ and the restriction of $\pi$ is
  injective. Conversely, if $\alpha$ is non-compact then the unique
  fundamental weight $\omega$ with
  $\langle\omega,\alpha^\vee\rangle=1$ is a weight of $P$. Moreover,
  by inspection of the table one sees that there is $d\ge1$ with
  $\omega\pm d\epsilon\in\overline\cM$. Thus the restriction of $\pi$
  is not injective.

  Finally, the lemma is proved by simply finding all real forms of
  $(\gf_\C,\fhq_\C)$, for which $\alpha$ is a compact root of $\gf$.
  For this we use Berger's list together with Table \ref{non-symmetric
    real forms}.  For example the first item yields among others the
  pair $(\sl(n,\HH),\sl(n,\C))$ which is real spherical if and
  only if $\alpha_n$ is compact for $\sl(n,\HH)$, hence if and only
  if $n$ is odd.
\end{proof}

\subsection{The classical cases}

\begin{proposition}\label{classification su(p,q)} 
  Let $\gf=\su(p,q)$, $1\leq p\leq q$, and let $\hf$ be a reductive
  subalgebra.  Then $(\gf,\hf)$ is real spherical if and only if
  either it is absolutely spherical or $\hf$ is conjugate to one of
  the following: {\renewcommand{\theenumi}{\roman{enumi}}
    \begin{enumerate}
    \item\label{h1h2}
      $\hf=\hf_1\oplus\hf_2=\su(p_1,q_1)\oplus\su(p_2,q_2)$ with
      $p_1+p_2=p, q_1+q_2=q$ and $p_1+q_1=p_2+q_2$, but
      $(p_1,q_1)\neq(q_2,p_2)$, or
    \item\label{1qsp} $p=1$, $q=q_1+q_2$, $q_2$ even and
      $\hf=\su(1, q_1) \oplus \symp(q_2/2)\oplus \ff$ with
      $\ff\subset \uf(1)$.
    \end{enumerate}}
\end{proposition}

\begin{proof} Since we shall refer to Proposition \ref{tower-factor}
  it will be convenient to replace the notation $\hf$ in the above
  statement by $\hf'$, and let $\hf$ instead denote a maximal
  reductive subalgebra with $\hf'\subset\hf\subset\gf$. Then $\hf$ is
  symmetric by Lemma \ref{most maximals are symmetric}.

  We need to consider the cases (1), (3) and (5) from Table
  \ref{lcaph1}.  For case (1) we observe that $\lf_\hf$ is compact but
  not $\hf$. Hence by Lemma \ref{factor compact} there is no proper
  real spherical subalgebra $\hf'$ of $\hf$.  We can argue similarly
  for (3) as symplectic algebras do not admit factorizations by
  Proposition \ref{Oni2}.

  This leaves us with (5), i.e.
  \[\hf=\hf_1 \oplus \hf_2\oplus \hf_3= \su(p_1,q_1) \oplus
    \su(p_2,q_2) \oplus \uf(1)\] with $p_1+q_1, p_2+q_2>0$ and
  $p_1+p_2=p$, $q_1+q_2=q$. Set $r:=p_1 +q_1$ and $s:=p_2 +q_2$.  We
  may assume that $r\leq s$.  Note that since $p\leq q$ this implies
  that $q_2-p_1\geq |p_2-q_1|$.

  According to Table \ref{lcaph1} we have
  \begin{equation} \label{supq1}\lf_\hf=\sf[\uf(p_2-q_1,
    q_2-p_1)\oplus \uf(1)^r]\, \end{equation} when $p_2\geq q_1$, and
  when $p_2\leq q_1$ we have
  \begin{equation}\label{supq2} \lf_\hf = \sf[\uf(q_1-p_2)\oplus
    \uf(q_2 -p_1)\oplus \uf(1)^p].\end{equation}

  Let us first consider the case where $\hf'\neq [\hf,\hf]$ and start
  with $p_2\geq q_1$.  The embedding into $\hf$ of \eqref{supq1} is
  such that the projection of $\lf_\hf$ to $\hf_2$ is injective. Hence
  we deduce from \eqref{factor condition} and Proposition \ref{Oni2}
  that $r= p_1+q_1=1$ and $\hf'=\sp(\frac{p_2}{2}, \frac{q_2}{2})$
  or $\hf'=\sp(\frac{p_2}{2}, \frac{q_2}{2})\oplus \uf(1)$ both of
  which are absolutely spherical according to Table \ref{non-symmetric
    real forms}.

  Next we consider the case where $p_2\leq q_1$ with $\lf_\hf$ given
  by \eqref{supq2}. Note that $\uf(1)^p$ projects injectively to both
  factors $\hf_1$ and $\hf_2$. Hence we deduce from \eqref{factor
    condition} and Proposition \ref{Oni2} that $p= p_1+p_2=1$.
  Without loss of generality let $p_1=1$ and $p_2=0$,
  i.e.~$\gf=\su(1,q)$ and
  $\hf=\su(1,q_1) \oplus \su(q_2) \oplus\uf(1)$.  Proposition
  \ref{Oni2} forces $q_2$ to be even and shows that
  $\hf'=\hf_1 \oplus \sp(q_2/2)$ or
  $\hf'=\hf_1 \oplus \sp(q_2/2)\oplus \uf(1)$, both of which are
  real spherical. This is case \ref{1qsp}.

  Finally let us consider the case where $\hf'=[\hf,\hf]=\hf_1+\hf_2$.
  According to Table \ref{non-symmetric real forms} this is absolutely
  spherical provided that $r<s$. For $r=s$, case \ref{h1h2} follows
  from Lemma \ref{center removed}.
\end{proof}

\begin{proposition} \label{classification su*} Let $\gf=\sl(m,\Hb)$
  for $m\geq 3$, and let $\hf$ be a reductive subalgebra.  Then
  $(\gf,\hf)$ is a real spherical pair if and only if it is absolutely
  spherical or, up to conjugation,
  \begin{enumerate}
  \item $\hf=\sl(m-1,\Hb)\oplus \ff $ with $\ff \subset \C$, or
  \item $\hf=\sl(m,\C)$ with $m$ odd.
  \end{enumerate}
\end{proposition}

\begin{proof} We need to treat the cases (4) and (6) from Table
  \ref{lcaph1}.  Now, since symplectic algebras do not admit
  factorizations, we are left with the two cases in (6).

  We begin with $\hf=\sf(\gl(m_1,\Hb)\oplus \gl(m_2,\Hb))$,
  $m=m_1+m_2$, $m_1\geq m_2$.  Set $\hf_1=\sl(m_1,\Hb)$,
  $\hf_2=\sl(m_2,\Hb)$ and $\hf_3=\zf(\hf)=\gl(1,\R)$.

  Here we have from Table \ref{lcaph1}
  \[\lf_\hf=\sf[\gl(m_1 -m_2,\Hb) \oplus \gl(1,\Hb)^{m_2}]\]
  with $\sl(m_1-m_2,\Hb)\subset \hf_1$ and $\gl(1,\Hb)^{m_2}$
  diagonally embedded. According to \cite{Kr}, $[\hf,\hf]$ is
  absolutely spherical if and only if $m_1\neq m_2$.  If $m_1=m_2$,
  then $\lf_\hf$ does not surject to the center of $\hf$ and hence
  $[\hf,\hf]$ is not spherical.  If $m_2>1$, then we obtain via
  \eqref{factor condition} and Proposition \ref{Oni2} that the only
  possible spherical subalgebra contained in $\hf$ is $[\hf,\hf]$.

  In case $m_2=1$, $\mf\cap\hf$ surjects onto $\hf_2=\su(2)$ and we
  obtain the cases listed in (1).

  The second possibility for $\hf$ is $\hf=\uf(1)\oplus
  \sl(m,\C)$. For that we first note that the dimension bound
  excludes $\hf':=\uf(1) \oplus \hf_2'$ with $\hf_2$ a proper
  reductive subalgebra of $\sl(m,\C)$ to be real spherical.  The
  cases where $[\hf,\hf]$ are spherical are deduced from Lemma
  \ref{center removed}.
\end{proof}

\begin{proposition}\label{classification so*}
  Let $\gf=\so^*(2m)$ for $m\geq 5$. Then a reductive subalgebra is
  real spherical if and only if it is absolutely spherical or
  conjugate to one of the following:
  \begin{enumerate}
  \item $\hf=\so^*(2m-2)$, or
  \item $m=5$, $\hf=\spin(5,2)$ or $\spin(6,1)$.
  \end{enumerate}
\end{proposition}

\begin{proof} We need to consider the cases (8) and (10) from Table
  \ref{lcaph1}.  In case of (8) there are no proper real spherical
  subalgebras of $\gf$ contained in $\hf$ by \eqref{factor condition}
  and Proposition \ref{Oni2}.  In case (10) with $\hf=\so(m,\C)$ the
  dimension bound excludes a proper reductive subalgebra of $\hf$ to
  be real spherical.

  Finally we need to treat the case where
  $\hf=\hf_1 \oplus\hf_2= \so^*(2m_1)\oplus \so^*(2m_2)$ with
  $m_1\leq m_2$, $m_1+m_2=m$. Here
  $\lf_\hf=\so^*(2(m_2-m_1))\oplus \so(2)^{m_1}$ with
  $\so^*(2(m_2-m_1))\subset\so^*(2m_2)$. When $m_1>1$ we deduce from
  Propositions \ref{tower-factor} and \ref{Oni2} that there exist no
  proper reductive subalgebras of $\hf$ which are real spherical.  If
  $m_1=1$, then $\hf=\so(2,\R) \oplus \so^*(2m_2)$, $\mf\cap\hf$
  surjects onto $\so(2,\R)$ and thus $\so^*(2m_2)$ is real spherical.
  Further factorizations are only possible for $m_2=4$ which results
  in the real spherical subalgebras $\spin(6,1)$ and $\spin(5,2)$ in
  $\hf_2=\so^*(8)\simeq \so(6,2)$.
\end{proof}

\begin{proposition}\label{classification so(p,q)}
  Let $\gf=\so(p,q)$, $1\leq p\leq q$, $p+q\geq 6$. Then a reductive
  subalgebra is real spherical if and only if it is either absolutely
  spherical or conjugate to one of the following:
  \begin{enumerate}
  \item\label{un} $p$, $q$, and $\frac{p+q}{2}$ all even, $p\neq q$,
    and $\hf=\su(\frac{p}{2},\frac{q}{2})$.
  \item\label{deux} $p=2r$, $q=2s+1$, $r\neq s$, and $\hf=\su(r,s)$.
  \item\label{trois} $p=2r+1$, $q=2s$, $r\neq s$, and $\hf=\su(r,s)$.
  \item\label{quatre} $p=1$, $q=q_1+4$ and $\hf=\so(1,q_1) +\hf'$ with
    $\hf'\subsetneq \so(4)\simeq \so(3)\times \so(3)$ a subalgebra
    such that $\hf'+ \diag\so(3) =\so(4)$.
  \item\label{cinq} $p=1$, $q=q_1+q_2$, $q_2\geq 5$, and
    $\hf = \so(1, q_1) \oplus \hf_2$ with $\hf_2\subsetneq \so(q_2)$ a
    subalgebra such that $\hf_2 +\so(q_2-1)=\so(q_2)$ (see Proposition
    \ref{Oni2}).
  \item\label{six} $p=2$, $q=q_1+7$ and
    $\hf = \so(2, q_1) \oplus \sG_2$.
  \item\label{sept} $p=2$, $q=q_1 +8$ and
    $\hf = \so(2, q_1) \oplus \spin(7)$.
  \item\label{huit} $p=3$, $q=q_1+8$ and
    $\hf=\so(3,q_1) \oplus \spin(7)$.
  \item\label{neuf} $p=3$, $q=6$ and $\hf=\so(2) \oplus \sG_2^1$.
  \item\label{dix} $p=4$, $q=7$ and $\hf=\so(3)\oplus \spin(3,4)$.
  \end{enumerate}
\end{proposition}

\begin{proof} According to Lemma \ref{typeII-so} and Lemma
  \ref{typeIII-so} there are no maximal spherical subalgebras of type
  II or III unless $\gf$ is split. Moreover, by Lemma \ref{typeI-so}
  all maximal subalgebras of type I are symmetric. As we may assume
  that $\gf$ is not quasisplit, we need only to consider subalgebras
  of the symmetric subalgebras $\hf$ from cases (7) and (9) in Table
  \ref{lcaph1}.

  We begin with the first of these,
  i.e.~$\hf=\uf(\frac{p}{2},\frac{q}{2})$ with $p\neq q$ even and
  $\lf_\hf =\uf(\frac{q-p}{2}) \oplus \sl(2,\R)^{\frac{q}{2}}$
  where $\uf(\frac{q-p}{2})\subset \so(q-p)=\mf$. In particular,
  $\uf(\frac{q-p}{2})$ surjects onto the center of $\hf$ and we deduce
  that $[\hf,\hf]$ is spherical. According to Krämer this is
  absolutely spherical if and only if $\frac{p+q}{2}$ is odd, and we
  obtain \ref{un}. From the structure of $\lf_\hf$ we deduce from
  Proposition \ref{Oni2} and Proposition \ref{tower-factor} that no
  other reductive proper subalgebra $\hf'\subset\hf$ is real spherical
  in $\gf$.

  This leaves us with the case (9) from Table \ref{lcaph1}, where
  $\hf=\hf_1 \oplus \hf_2=\so(p_1,q_1) \oplus \so (p_2,q_2)$ with
  $p=p_1+p_2\leq q=q_1+q_2$, $r=p_1+q_1\le s=p_2+q_2$, and
  $p_1\leq q_2$.  In case $p_2\leq q_1$ we have $P=Q$ and
  \begin{equation}\label{sopq2} \lf_\hf = \so(q_1-p_2)\oplus \so(q_2
    -p_1)\, .\end{equation}
  In case $p_2\geq q_1$ we have  
  \begin{equation} \label{sopq1}\lf_\hf=\so(p_2-q_1,
    q_2-p_1)\subset\hf_2\, .\end{equation}

  To start with we exclude the diagonal case where $\hf_1\simeq \hf_2$
  and $\hf'\simeq \hf_1$ is ``diagonally'' embedded into
  $\hf=\hf_1 \oplus \hf_2$. For that we note that $\hf_1\simeq \hf_2$
  either means that $\hf_1=\hf_2$ or $(p_1,q_1)=(q_2,p_2)$.  In the
  latter case $\gf$ is split and we are left with $\hf_1=\hf_2$.  In
  particular $\hf$ is non-compact and semisimple, but
  $\lf_\hf=\so(q_1-p_1)\oplus \so(q_1 -p_1)$ is compact.  Thus
  $H=H'L_H$ is not possible by Lemma \ref{factor compact}.

  We now begin with the case of $p_2\leq q_1$ and \eqref{sopq2}.
  Suppose that $\hf_1\oplus \hf_2'$ is a spherical subalgebra of
  $\gf$, with $\hf'_2\subsetneq \hf_2$.  Then
  $\hf_2=\hf_2' +\lf_\hf^2$ with $\lf_\hf^2=\so(q_2-p_1)$, i.e.
  \begin{equation}\label{factorization of h2}
    \hf_2'+ \so(q_2-p_1)  = \hf_2= \so(p_2,q_2)\, .
  \end{equation}

  Suppose first that $\hf_2$ is simple. Then this is a factorization
  of of $\so(p_2,q_2)$ with one factor compact and we deduce from
  Lemma \ref{factor compact} that $\so(p_2,q_2)$ is compact as well,
  i.e.~$p_2=0$ or $q_2=0$.  If $q_2=0$, then $p_1=0$, and
  $\hf_2'=\hf_2$.  Hence we may assume that $p_2=0$ and then
  \[\hf_2'+\so(q_2-p_1) =\so(q_2)\, .\]
  We deduce from Proposition \ref{Oni2} that $p_1=1,2,3$ and read off
  the possibilities \ref{cinq}, \ref{six}, \ref{sept}, \ref{huit}
  for $\hf_2'$.

  In case $\hf_2$ is not simple possibilities are $\hf_2=\so(2,2)$ and
  $\hf_2=\so(4)$. Only the latter is possible with \eqref{factorization
    of h2}, which in that case gives $p_1=1$ and
  $\hf_2'\simeq \so(3)$ and leads to \ref{quatre}.

  The case where $\hf_1'\oplus \hf_2$ with $\hf_1'\subsetneq \hf_1$ is
  analogous, and leads to the same results but with $r$ and $s$
  interchanged.  This finishes the treatment of $p_2\leq q_1$ and
  \eqref{sopq2}.

  We now treat the case of $p_2\geq q_1$ and \eqref{sopq1}, where
  $\lf_\hf\subset\hf_2$.  Let $\hf':= \hf_1\oplus \hf_2'\subset \hf$
  with $\hf_2'\subset\hf_2$ be a spherical subalgebra of $\gf$.  The
  condition is that $H_2/H_2'$ is a real spherical space for the
  action of $L_H$.  In particular we must have (see Corollary \ref{cor
    tower-factor})
  \begin{equation} \label{sopq4} \hf_2' + \so(p_2 -q_1,q_2-p_1) =
    \hf_2= \so(p_2,q_2)\, \end{equation} and $p_2$, $q_2$ both
  non-zero.  Hence $\hf_2$ is non-compact and we may assume it is
  simple.  According to Proposition \ref{Oni2} we derive that
  $p_1+q_1$ equals $1,2,3$.  Suppose first that $\hf_2'$ is of Type I
  in $\hf_2$, see the four cases in Lemma \ref{typeI-so}.  Onishchik's
  list, Table \ref{Onishchik}, shows that only
  $\hf_2'=\uf(\frac {p_2}{2},\frac{q_2}{2})$ can be compatible with
  \eqref{sopq4}, and then $p_1+q_1=1$.  Hence $\hf=\hf_2$ and
  $\hf'=\hf_2'$ is real spherical according to Krämer.  Further
  subalgebras of type $\hf'':=\uf(1)\oplus \hf_2''$ with
  $\hf_2''\subset[\hf_2',\hf_2']$ a proper maximal spherical
  subalgebra are excluded. Indeed \eqref{sopq4} and Proposition
  \ref{Oni2} only allow $\hf_2''=\sp(\frac{p_2}{4}, \frac{q_2}{4})$
  and we arrive at the tower
  \[\gf=\so(p_2+1,q_2) \supset \uf(\frac{p_2}{2},\frac{q_2}{2})\supset
    \hf''=\sp(\frac{p_2}{4}, \frac{q_2}{4})+\uf(1)\, .\] Now the
  real spherical pair
  $(\gf,\hf)=(\so(p_2+1, q_2), \uf(\frac{p_2}{2},\frac{q_2}{2}))$ has
  structural algebra $\lf_\hf=\uf(\frac{|q_2-p_2|}{2})$ and we deduce
  that $(\gf, \hf'')$ is not real spherical by Proposition
  \ref{tower-factor} and Proposition \ref{Oni2}.  This leads us to
  decide whether
  $\hf_3:=[\hf_2',\hf_2']= \su(\frac {p_2}{2},\frac{q_2}{2})$ is real
  spherical.  According to Krämer, $\hf_3$ is not absolutely
  spherical. Without loss of generality we may assume $p_1=1$ and
  $q_1=0$ and then $\lf_\hf=\so(p_2, q_2-1)$. Observe that $\hf_3$ is
  real spherical if and only if $\hf\cap \mf$ surjects onto the center
  of $\hf_2'$.  This is the case precisely when $p_2\neq q_2$ and
  $p_2\neq q_2-1$ (see Lemma \ref{center removed}), and it leads to
  cases \ref{deux}--\ref{trois}.

  Finally assume that $\hf_2'$ is of type II or III in $\hf_2$.  The
  type II subalgebras appear only for $\hf_2=\so(4,4)$ (see Lemma
  \ref{typeII-so}) and are excluded by the dimension bound
  \eqref{dimso}.  This leaves us with the examination of the two type
  III cases from Lemma \ref{typeIII-so}.  We begin with
  $\hf_2'=\spin(4,3)$ in $\hf_2=\so(4,4)=\so(p_2,q_2)$. Recall that
  $p_1+q_1=1,2,3$ and note that $p_1$ and $q_1$ have to differ by
  three in order for $\gf$ not to be quasisplit. Hence we may assume
  that $\hf_1=\so(3)$, i.e.~$\gf=\so(4,7)$ and
  $\hf'=\so(3)\oplus \spin(4,3) \subset\so(3)\oplus \so(4,4)$.

  We claim that $\hf'$ is real spherical.  For that we need to show
  that the $L_H=\SO_0(1,4)$-space $H_2/H_2'=\SO_0(4,4)/\Spin(4,3)$ is
  real spherical. To this end we lift to $\Spin(4,4)$, apply the
  exceptional outer automorphism which swaps the simple roots
  $\alpha_1$ and $\alpha_4$, and go back to $\SO_0(4,4)$. Then
  $\Spin(4,3)$ and $\SO_0(1,4)\subset \SO_0(2,4)$ are converted to
  $\SO_0(4,3)$ and $\Sp(1,1)\subset \SU(2,2)$, respectively. The
  complexification of this situation is the third case of Table
  \ref{Onishchik} with $n=2$. Using the last column of the table we
  get $\SO_0(4,4)/\SO_0(4,3)=\Sp(1,1)/\Sp(0,1)$, which real spherical
  as a $\Sp(1,1)$-variety by Lemma \ref{center removed}.  This proves
  the claim and furnishes case \ref{dix}.

  Next we move on to the case where $\hf_2'=\sG_2^1$ and
  $\hf_2=\so(3,4)$.  As before $p_1 +q_1=1,2,3$.  The case $p_1+q_1=3$
  is excluded by the dimension bound and the case with $p_1 +q_1=1$
  leads to absolutely spherical pairs.  The case $p_1=q_1=1$ is
  quasi-split. This leaves us with $\hf'=\so(2)\oplus \sG_2^1$ in
  $\gf=\so(3,6)$.

  We claim that this case is real spherical.  Here we have to show
  that $H_2/ H_2'=\SO_0(3,4)/\sG_2^1$ is spherical as
  $L_H=\SO_0(1,4)$-variety. But that follows immediately from the
  isomorphism $\SO_0(3,4)/\sG_2^1\cong\SO_0(4,4)/\Spin(3,4)$ (eighth
  case of Table \ref{Onishchik}) and the proof of case \ref{dix}
  above. This yields case \ref{neuf}.
\end{proof}

\begin{proposition}\label{classification sp(p,q)} 
  Let $\gf=\symp(p,q)$ and let $\hf$ be a reductive subalgebra. Then
  $\hf$ is real spherical if and only if it is absolutely spherical or
  conjugate to one of the following:
  \begin{enumerate}
  \item $\sp(p-1,q)$,
  \item $\sp(p,q-1)$,
  \item $\su(p,q)$ with $p\neq q$.
  \end{enumerate}
\end{proposition}

\begin{proof} We need to consider subalgebras of the following cases
  from (11) and (12) in Table \ref{lcaph1}:
  \[\hf=\symp(p,\C), \ \symp(p_1, q_1) \times \symp(p_2,q_2), \
    \uf(p,q), \ \gl(p,\Hb) \] where $q=p$ in the first and last cases.
  Since symplectic algebras admit no factorizations by Proposition
  \ref{Oni2} the first case is excluded with Proposition
  \ref{tower-factor}.  We can argue similarly in the second case,
  except when $\lf_\hf$ surjects onto one of the factors of
  $\hf=\hf_1\oplus\hf_2$.  According to Table \ref{lcaph1} this
  happens if and only if $\hf_1$ or $\hf_2$ is $\symp(1)$, in which
  case the other factor $\hf'$ of $\hf$ is $\sp(p-1,q)$ or
  $\sp(p,q-1)$.  Then $\mf$ belongs to $\hf$ and surjects onto
  $\symp(1)$ along $\hf'$. Hence $\hf'$ is spherical. Further we
  observe that it is not absolutely spherical, but any strictly larger
  subalgebra is.  This gives (1)-(2).

  For the third case, $\hf=\uf(p,q)$, we note that $\lf\cap\hf$ is
  compact according to Table \ref{lcaph1}. Hence if $\hf'\subset \hf$
  satisfies \eqref{factor condition} then Lemma \ref{factor compact}
  implies $\su(p,q)\subset\hf'$.  With Lemma \ref{center removed} we
  conclude (3).

  Finally, in the fourth case $\hf=\gl(p,\Hb)$ we have
  $\lf_\hf =\uf(1)^p$ and hence no proper factorization is possible.
\end{proof}

\subsection{The exceptional cases}

For convenience we record here Cartan's list of the nine symmetric
subgroups in the complex exceptional Lie groups of type E:

\smallskip
\noindent
\begin{tabular}{l|c|c|c}
  $G_\C$&$\sE_6^\C$&$\sE_7^\C$&$\sE_8^\C$\\\hline
  $H_{1\C}$&$\Sp(4,\C)$&$\SL(8,\C)$&$\SO(16,\C)$\\
  $H_{2\C}$&$\SL(6,\C)\times\SL(2,\C)$&$\SO(12,\C)\times\SL(2,\C)$&$\sE_7^\C\times\SL(2,\C)$\\
  $H_{3\C}$&$\SO(10,\C)\times\SO(2,\C)$&$\sE_6^\C\times\SO(2,\C)$&\\
  $H_{4\C}$&$\sF_4^\C$&&\\
\end{tabular}

\smallskip For the list of real symmetric subgroups we shall refer to
\cite{Berger}.

\begin{proposition}\label{classification exceptional} 
  Let $\gf$ be a non-complex and non-compact simple exceptional Lie
  algebra and let $\hf\subset\gf$ be reductive.  Then $(\gf, \hf)$ is
  real spherical if and only if it is absolutely spherical or, up to
  conjugation,
  \begin{enumerate}
  \item $\gf=\sF_4^2$ and $\hf=\symp(2,1)\oplus \ff $ with
    $\ff \subset \uf(1)$, or
  \item $\gf=\sE_6^4$ and $\hf=\sl(3,\Hb)\oplus \ff$ with
    $\ff\subset\uf(1)$, or
  \item $\gf=\sE_7^2$ and $\hf=\sE_6^2$ or $\sE_6^3$.
  \end{enumerate}
\end{proposition}

\begin{proof} If $\gf$ is quasi-split we apply Lemma \ref{lemma
    quasi-split}. In particular we can then assume
  $\gf_\C\neq \sG_2^\C$.

  This leaves us with the $\sE$ and $\sF$-cases which are not
  quasi-split.  As before we shall use the notation $\hf'$ for a given
  candidate of a real spherical subalgebra of $\gf$.  It follows from
  Lemma \ref{max is symmetric} that we may assume $\hf'$ is contained
  in a symmetric subalgebra which we then denote by $\hf$. We can
  assume the inclusion is proper.

  We use Table \ref{lcaph2} for $\lf\cap \hf$, where $Q=LU$ is the
  adapted parabolic for $Z=G/H$. We note that $\lf\cap \hf$ only
  depends on the real form $\gf$ and the complexification $\hf_\C$
  (see Lemma \ref{lcaph}).

  We start with $\gf=\sF_4^2$. From Table \ref{lcaph2} we deduce that
  either $\hf=\so(8,1)$ and $\lf_\hf=\so(7)$, or
  $\hf=\sp(2,1) \oplus \su(2)$ and $\lf_\hf=\so(4)\oplus \so(3)$.
  Since $\so(9,\C)$ does not admit non-trivial factorizations, no
  proper reductive subalgebra of $\hf$ can be spherical in the first
  case.  In the second case we observe with Remark \ref{lcaph
    embedding} that $\lf_\hf$ projects onto $\su(2)$, the second
  factor of $\hf$, and the cases in (1) emerge.

  We continue with $\gf=\sE_8^2$. Any symmetric subalgebra of $\gf$ is
  a real form of either $\hf_{1,\C}$ or $\hf_{2,\C}$ from the table
  above.  However, no proper reductive subalgebra of $\hf_{1}$ can
  satisfy the dimension bound \eqref{dimE_8^2}.  We assume
  $\hf'\subsetneq\hf_2$, a real form of $\hf_{2,\C}$, and let $Q=LU$
  be the adapted parabolic for $G/H_2$. Then
  \begin{equation}\label{dec hf2}
    \lf\cap \hf_2 + \hf'=\hf_2
  \end{equation}
  by \eqref{factor condition}.  Here $(\lf\cap \hf_2)_\C= \so(8,\C)$
  and the projection of this algebra to the second component of
  $\hf_{2,\C}$ must be zero as $\so(8,\C)$ is irreducible and of
  higher dimension than $\sl(2,\C)$.  Hence
  $(\lf\cap \hf_2)_\C\subset \sE_7^\C$. However by Proposition
  \ref{Oni2}, $\sE_7^\C$ admits no proper factorizations, and hence we
  must have $\sE_7^\C\subset \hf'$.  Thus $\lf\cap \hf_2\subset \hf'$
  which contradicts \eqref{dec hf2}.

  Next we investigate the real forms of $\gf_\C=\sE_7^\C$.  According
  to the table above the symmetric subalgebras in $\gf_\C$ are
  $\hf_{i,\C}$, $i=1,2,3$.

  We start with $\gf=\sE_7^2$ and note that the dimension bound
  \eqref{dimE_7^2} excludes that $\hf\subsetneq \hf_1$.  Next we
  consider the pair $(\gf,\hf_2)$ and the corresponding adapted
  parabolic $Q=LU$. Assume $\hf\subsetneq\hf_2$, then \eqref{dec hf2}
  holds as before.  Here $(\lf\cap \hf_2)_\C= \sl(2,\C)^3$ embeds
  in the first component $\so(12,\C)$ of $\hf_{2,\C}$ (see Remark
  \ref{lcaph embedding}).  By Proposition \ref{Oni2} there exists no
  proper factorization of $\so(12,\C)$ with $\sl(2,\C)^3$ as a
  factor, and hence we conclude that $\so(12,\C)\subset \hf_\C$.  Thus
  $\lf\cap \hf_2\subset \hf$ which contradicts \eqref{dec hf2}.

  For the third case we note that $\sE_6^\C$ has codimension 1 in
  $\hf_{3,\C}$ and does not admit proper factorizations by Proposition
  \ref{Oni2}. Hence if $\hf\subsetneq\hf_3$ then $\hf_\C=\sE_6^\C$.
  We deduce that $\hf$ is real spherical from Lemma \ref{center
    removed}. This gives (3).

  We move on with $\gf=\sE_7^3$, and start with the assumption that
  $\hf\subsetneq \hf_1$.  Here $Q=P$ by Lemma \ref{adapQ} and this
  implies that $\dim H_1/H_1\cap L=\dim G/P= 51$, and hence
  $\dim H_1\cap L=12$.  According to Proposition \ref{Oni2} the only
  factorizations for $\hf_1$ are given by
  \begin{itemize}
  \item $(\sl(8,\C), \symp(4,\C), \sl(7,\C))$,
  \item
    $(\sl(8,\C), \symp(4,\C), \sf(\gl(1,\C) \oplus \gl(7,\C)))$,
  \end{itemize}
  and none of these factors have dimension $12$. With Proposition
  \ref{tower-factor} we reach a contradiction.

  For the case of $\hf_2$ we first recall
  $\hf_2\cap\lf = \so(6) \oplus \so(2) \oplus \sl(2,\R)$.  It
  follows from Remark \ref{lcaph embedding} that $\hf_2\cap\lf$ does
  not surject onto the $\sl(2)$-component of $\hf_2$.  Hence no
  proper reductive subalgebras of $\hf_2$ can be real spherical.

  For $\hf_3$ we are again left to check whether a real form of the
  first component $\sE_6^\C$ of $\hf_{3,\C}$ is real spherical.  Here
  $(\lf\cap \hf_2)_\C= \so(8,\C)$ and thus the projection of
  $(\lf\cap \hf_2)_\C$ to the second component of $\hf_3$ is trivial,
  and hence no real form of $\sE_6^\C$ can be spherical.

  Finally we consider $\gf_\C=\sE_6^\C$ and $\gf=\sE_6^3$ or
  $\sE_6^4$.  The complex symmetric subalgebras $\hf_{1,\C}$,
  $i=1,\dots,4$, are given in the table.

  Since both $\hf_1$ and $\hf_4$ admit no factorizations there are no
  reductive real spherical subalgebras which are contained in $\hf_1$
  or $\hf_4$. We move on and assume $\hf\subsetneq \hf_{2}$, where
  $\hf_2$ is a real form of $\hf_{2,\C}$ in $\gf$.  Write
  $\hf_2=\hf_2' \oplus \hf_2''$ with $\hf_{2,\C}'=\sl(6,\C)$ and
  $\hf_{2,\C}''=\sl(2,\C)$.  {}From Table \ref{lcaph2} we infer
  that
  \begin{equation} \lf\cap \hf_2 =\uf(2) \oplus \uf(2) \quad
    \hbox{for} \quad \gf=\sE_6^3
  \end{equation} 
  \begin{equation} \lf\cap \hf_2 =\so(5) \oplus \so(3)\oplus \gl(1,\R)
    \quad \hbox{for} \quad \gf=\sE_6^4
  \end{equation} 
  We claim that $\hf$ is not spherical for $\gf=\sE_6^3$. Otherwise,
  according to Propositions \ref{Oni2}, \ref{tower-factor},
  $\lf_{\hf_2}$ would surject to a factor of $\hf_2''$.  But this is
  not possible by Remark \ref{lcaph embedding}.

  For $\gf=\sE_6^4$, we claim that $\lf\cap \hf_2$ surjects onto
  $\hf_2''=\su(2)$ and in particular that $\hf_2'=\sl(3,\Hb)$ is
  real spherical.  In order to establish that we let $V\subset \gf_\C$
  be the orthogonal complement of $\hf_\C$ in $\gf_\C$. Note that
  $\dim_\C V=40$ and that $V$ is an irreducible module for
  $\hf_\C$. Hence $V = \bigwedge^3 \C^6 \otimes \C^2$ as an
  $\hf_{2,\C}$-module.  Notice that $\af:=V\cap\zf(\lf)\neq \{0\}$ and
  that $\af$ is fixed under $\lf_\hf$. In order to obtain a
  contradiction, assume that $\lf_{\hf_2}\subset \hf_2'$. Then, as an
  $\lf_\hf$-module, $V = \bigwedge^3 \C^6\oplus \bigwedge^3
  \C^6$. Since $V^{\lf_\hf}\neq \{0\}$ we deduce that the irreducible
  $\hf_{2,\C}'=\sl(6,\C)$-module $\bigwedge^3 \C^6$ is spherical
  for
  $[\lf\cap\hf_2, \lf\cap\hf_2]_\C\simeq \sp(2,\C) \oplus
  \sp(1,\C)$.  But $\bigwedge^3\C^6$ decomposes under
  $\sp(3,\C)$ into $V(\omega_1)\oplus V(\omega_3)$ and hence is not
  spherical for the pair
  $(\sp(3,\C),\sp(2,\C) \oplus\sp(1,\C))$ by
  \cite{Kr}*{Tabelle 1}.  This gives the desired contradiction, and
  hence (2).

  Finally we come to the case of $\hf_{3,\C}$.  Here it is known that
  $\hf_{3,\C}'=\so(10,\C)$ is a complex spherical subgroup by
  \cite{Kr}.  {}From the list in Proposition \ref{Oni2} we extract
  that the only factorizations of $\so(10,\C)$ are given by
  $(\so(10,\C), \so(9,\C), \sl(5,\C)+\ff)$, $\ff\subset\uf(1)$.
  Now for $\gf=\sE_6^3, \sE_6^4$ we have
  $[\lf\cap\hf_3,\lf\cap\hf_3]_\C$ is $\so(6,\C)$ or $\spin(7,\C)$ and
  the factorization of $\hf_{3,\C}$ is not possible. This concludes
  the proof of the proposition.
\end{proof}

By combining Propositions \ref{classification su(p,q)},
\ref{classification su*}, \ref{classification so*},
\ref{classification so(p,q)}, \ref{classification sp(p,q)}, and
\ref{classification exceptional} we finally obtain Table~\ref{table
  real spherical}.

\section{Absolutely spherical pairs}

In this section we prove Theorem \ref{theorem classification}.  For
that it only remains to classify the absolutely spherical pairs, and
we refer to \cite{Berger} for the symmetric ones.

\subsection{The complex cases} We begin by determining the cases for
which $\gf$ has a complex structure.

\begin{proposition} \label{classification complex}Let $\gf$ be a
  complex simple Lie algebra and $\hf\subset\gf$ a real reductive
  subalgebra.  Then $(\gf,\hf)$ is real spherical if and only it is
  absolutely spherical. This is the case if and only if one of the
  following holds \renewcommand\labelenumi{(\roman{enumi})}
  \renewcommand\theenumi\labelenumi
  \begin{enumerate}
  \item\label{cc1} $\hf$ is a real form of $\gf$ (and hence
    symmetric),
  \item\label{cc2}$\hf$ is a complex spherical subalgebra of $\gf$,
  \end{enumerate}
  or $\hf$ is conjugate to $\hf_1\oplus \hf_2$ with
  \begin{enumerate}\setcounter{enumi}{2}
  \item\label{cc3} $(\hf_1,\hf_2)=(\zf(\hf),[\hf,\hf])$,
    $\dim_\R \hf_1=1$, and $(\gf,\hf_2)$ is one of the following
    complex spherical pairs
    \begin{enumerate}
    \item $(\sl(n+m,\C),\sl(n,\C) \times \sl(m,\C))$,
      $0< m< n$
    \item $(\sl(2n+1,\C),\sp(n,\C))$, $n\ge 2$
    \item $(\so(2n,\C),\sl(n,\C))$, $n$ odd, $n\ge 3$
    \item $(\sE_6^\C,\so(10,\C))$,
    \end{enumerate}
  \item \label{cc4}$\hf_1=\su(2)$ or $\sl(2,\R)$, and
    $(\gf,\hf_2)=(\sp(n+1,\C),\sp(n,\C))$, $n\ge 1$.
  \end{enumerate}
\end{proposition}

\begin{proof} First observe that $\gf$, considered as a real Lie
  algebra, is quasisplit.  Hence $\hf\subset\gf$ is real spherical if
  and only if $\hf_\C\subset\gf_\C$ is spherical.  We may identify
  $\gf_\C$ with $\gf \oplus \gf$. Now according to \cite{Mik}*{Section 5},
  the complex spherical subalgebras $\tilde \hf$ of $\gf\oplus\gf$ are
  given as follows

  {\renewcommand\labelenumi{(\roman{enumi})}
    \renewcommand\theenumi\labelenumi
    \begin{enumerate}
    \item $\tilde\hf = \diag (\gf)$ (cf. \cite{Mik}*{Prop.\ 5.4}).
    \item $\tilde\hf=\hf_1 \oplus \hf_2$ with $\hf_i\subset\gf$
      complex spherical.
    \item There exists a complex spherical subalgebra
      $\hf_0\subset\gf$ with $\zf(\hf_0)\neq 0$, $[\hf_0,\hf_0]$
      complex spherical and
      $\tilde \hf = [\hf_0,\hf_0] \oplus \zf(\hf_0) \oplus
      [\hf_0,\hf_0]$ and $\zf(\hf_0)$ diagonally embedded (see
      \cite{Mik}, beginning of Section 5 with the notion of a {\it
        principal irreducible spherical pair}).
    \item $\gf=\sp(n+1,\C)$ and
      $\tilde\hf = \sp(n,\C) \oplus \sp(1,\C)
      \oplus\sp(n,\C)$ with $\sp(1,\C)$ diagonally embedded and
      $n\ge1$ (cf. \cite{Mik}*{Prop.\ 5.4}).
    \end{enumerate}} When restricted to subalgebras of the form
  $\tilde\hf= \hf_\C$ these four cases correspond to the four cases
  listed in the proposition. This is easily seen, with use of Krämer's
  list for \ref{cc3}.
\end{proof}

In Table~\ref{non-symmetric spherical} at the end of this section we
record the list of the non-symmetric pairs of case \ref{cc2}.  The
pairs in \ref{cc3}-\ref{cc4} are tabulated in Table~\ref{complex
  forms}.

\begin{remark}\label{remark to table 6}
  Inspecting Table~\ref{non-symmetric spherical} one realizes that it
  has a certain structure (cf.~\cite{JWolf}*{Table (12.7.2)}). In all
  cases but $(8)$ and $(9)$ there is a canonical intermediate
  subalgebra $\fhq_\C$, given in the last column, with the following
  properties (a)-(c).
  \begin{itemize}

  \item[(a)] The pair $(\gf_\C,\fhq_\C)$ is symmetric. Hence all of
    its real forms appear up to isomorphism in Berger's list.
  
  \item[(b)] Except for case $(10)$, the pair $(\fhq_\C,\hf_\C)$ is
    symmetric, as well. Even though $\fhq_\C$ may not be simple the
    real forms of these pairs are easily read off from Berger's list.

  \item[(c)] If $(\gf_\C,\hf_\C)$ is defined over $\R$ then also
    $\fhq_\C$ is defined over $\R$. Indeed,
    $\fhq_\C=N_{\gf_\C}(\hf_\C)$ in cases $(1)$, $(4)$, and
    $(11)$. Moreover,
    $\fhq_\C=N_{\gf_\C}(C_{\gf_\C}([\hf_\C,\hf_\C]))$ in cases $(2)$,
    $(3)$, and $(7)$. For cases $(5)$ and $(6)$ one argues as follows:
    for all real forms of $\gf_\C$, the defining representation $V$ is
    defined over $\R$. Then $\fhq_\C=C_{\gf_\C}(V^{\hf_\C})$ is
    defined over $\R$, as well. In case $(10)$ the isomorphism
    $\so^*(8)\cong\so(6,2)$ (via triality) shows that it suffices to
    consider real forms for which $V$ is defined over $\R$. Then the
    argument above works.
  \end{itemize}
\end{remark}

\subsection{The non-complex cases} \label{non-complex} We recall that
$\gf$ carries no complex structure if and only if it remains simple
upon complexification. In this case we say that $\gf$ is {\it
  absolutely simple}.

Assume $\gf$ is non-compact and absolutely simple. Using the reasoning
in Remark \ref{remark to table 6}, we obtain all non-symmetric,
absolutely spherical reductive subalgebras $\hf$.  The list is given
in Table \ref{non-symmetric real forms} below.  Only the last five
rows, which relate to (8), (9) and (10) above, require a separate
argument.

The cases involving real forms of $\sG_2$ are handled using the
following remarks.  The maximal compact subalgebra of $\sG_2^1$ is
$\su(2)+\su(2)$. Hence $\su(3)\not\subset \sG_2^1$. Moreover, the
invariant scalar product on the $7$-dimensional representation of
$\sG_2$ and $\sG_2^1$ has signature $(7,0)$ and $(4,3)$,
respectively. In the second one the isotropy group of a vector with
positive or negative square length is $\SL(3,\R)$ or $\SU(2,1)$,
respectively. This gives the pairs related to (8) and (9), and finally
$(10)$ can be reduced to case $(9)$ in the same way as (b) above.

This completes the proof of Theorem \ref{theorem classification}.

\subsection{Tables}
\label{app tables}

Here we tabulate (up to isomorphism) all absolutely spherical
non-symmetric pairs $(\gf, \hf)$ with $\gf$  non-compact, simple and
$\hf\subset \gf$ reductive:
\begin{itemize}
\item[-]Table~\ref{non-symmetric spherical} lists those pairs in which
  both $\gf$ and $\hf$ have a complex structure.  In this table all
  algebras are implied to be complex.  The table is due to Krämer
  \cite{Kr}.

\item[-]Table~\ref{complex forms} lists those pairs in which $\gf$ but
  not $\hf$ has a complex structure.  The table is extracted from
  Proposition \ref{classification complex}.

\item[-]Table~\ref{non-symmetric real forms} lists those pairs in
  which $\gf$ is absolutely simple. See Section \ref{non-complex}.
\end{itemize}

\begin{table}[b]
  \[
    \begin{array}{rllll}
      &\gf_\C&\hf_\C&&\fhq_\C\\
      \noalign{\smallskip}
      \cline{2-5}
      \noalign{\smallskip}
      (1)&\sl(m+n)&\sl(m)+\sl(n)&m>n\ge 1&\sf[\gl(m)+\gl(n)]\\
      (2)&\sl(2n+1)&\sp(n)+\ff&n\ge2,\ff\subset \C&\gl(2n)\\
      (3)&\sp(n)&\sp(n-1)+\C&n\ge3&\sp(n-1)+\sp(1)\\
      (4)&\so(2n)&\sl(n)&n\ge5\text{ odd}&\gl(n)\\
      (5)&\so(2n+1)&\gl(n)&n\ge 2&\so(2n)\\
      (6)&\so(9)&\spin(7)&&\so(8)\\
      (7)&\so(10)&\spin(7)+\C&&\so(8)+\C\\
      (8)&\sG_2&\sl(3)&&-\\
      (9)&\so(7)&\sG_2&&-\\
      (10)&\so(8)&\sG_2&&\so(7)\\
      (11)&\sE_6&\spin(10)&&\spin(10)+\C
    \end{array}
  \]
  \centerline{\Tabelle{non-symmetric spherical}}
\end{table}

\begin{table}%[ht]
  \[
    \begin{array}{llll}
      \gf&\hf & \\
      \hline
      \sl(n+m,\C)& \sl(n,\C)+\sl(m,\C)+\zf &\zf \subset \C, \dim_\R \zf=1 & 0<m<n\\
      \sl(2n+1, \C) & \sp(n,\C) +\zf &\zf \subset \C, \dim_\R \zf=1 &    n\geq 2 \\ 
      \so(2n,\C) & \sl(n,\C)+\zf &\zf \subset \C, \dim_\R \zf=1 & n\geq 3,  n\text{ odd}\\
      \sE_6^\C &\so(10,\C)+\zf&\zf \subset \C, \dim_\R \zf=1 & \\ 
      \hline 
      \sp(n+1,\C) & \sp(n,\C) + \ff& \ff \in\{\sp(1), \sp(1,\R)\} & n\geq 1\\
      \hline 
    \end{array}
  \]
  \centerline{\Tabelle{complex forms}}
\end{table}

\begin{table}%[ht]
  \[
    \begin{array}{lll}
      \gf&\hf \\
      \hline
      \sl(m+n,\R)& \sl(m,\R)+\sl(n,\R)&m>n\ge1\\
      \su(p_1+p_2,q_1+q_2) & \su(p_1,q_1)+\su(p_2,q_2)&p_1 +q_1>p_2+q_2\ge1\\ 
      \sl(m+n,\HH)& \sl(m,\HH)+\sl(n,\HH)&m>n\ge1\\
      \sl(2n+1,\R)&\sp(n,\R)+\ff&n\ge2,\ff\subset \R\\
      \su(2p+1,2q)&\sp(p,q)+\ff& p+q\ge 2,\ff\subset  i\R\\
      \su(n+1,n)&\sp(n,\R)+\ff& n\ge2,\ff\subset  i\R\\
      \hline
      \sp(n,\R)&\sp(n-1,\R)+\ff&n\ge 2,\ff\in\{\R,i\R\}\\ 
      \sp(p,q)&\sp(p-1,q)+i\R&p,q\ge1\\ 
      \hline
      \so(2p,2q)  & \su(p,q)&p\ge q\ge1,\ p+q\ \hbox{odd}\\
      \so(n,n)  & \sl(n,\R)& n\ge3\ \hbox{odd}\\
      \so^*(2n)&\su(p,q)&n=p+q\ge3\ \hbox{odd} \\ 
      \so(2p+1,2q)&\su(p,q)+i\R&p+q\ge2\\ 
      \so(n+1,n)&\sl(n,\R)+\R& n\ge2\\ 
      \hline
      \so(5,4)&\spin(4,3)\\
      \so(8,1)&\spin(7,0)\\
      \so(5,5)&\spin(4,3)+\R \\
      \so(6,4)&\spin(4,3)+i\R\\
      \so(8,2)&\spin(7,0)+i\R\\
      \so(9,1)&\spin(7,0)+\R\\
      \so^*(10)&\multicolumn{2}{l}{\spin(6,1)+ i\R,\ \spin(5,2)+i\R }\\
      \hline
      \sE_6^1&\so(5,5)\\
      \sE_6^2&\so(6,4),\ \so^*(10)\\
      \sE_6^3&\so(10),\ \so(8,2),\ \so^*(10)\\
      \sE_6^4&\so(9,1)\\ 
      \hline 
      \sG_2^1&\sl(3,\R),\ \su(2,1)\\
      \so(4,3)&\sG_2^1\\
      \so(4,4)&\sG_2^1\\
      \so(5,3)&\sG_2^1\\ 
      \so(7,1)&\sG_2\\
      \hline
    \end{array}
  \]
  \centerline{\Tabelle{non-symmetric real forms}}
\end{table}

\end{document}